\documentclass[english,12pt]{article}

\pdfoutput=1

\usepackage{amssymb}
\usepackage{babel}
\usepackage{amsthm}
\usepackage{verbatim}
\usepackage{amsrefs}
\usepackage{amsmath}
\usepackage{listings}
\usepackage{fullpage}
\usepackage{graphicx}
\usepackage{qtree}
\usepackage[usenames,dvipsnames]{color}
\usepackage{array}
\usepackage{enumerate}
\usepackage{polynom}
\usepackage{tikz}
\usepackage{fancyhdr}
\usepackage{etoolbox}
\usepackage{algorithm}
\usepackage{algorithmic}

\newtheorem{defin}{Definition}
\newtheorem{prop}{Proposition}
\newtheorem{theorem}{Theorem}
\newtheorem{lemma}{Lemma}
\newtheorem{cor}{Corollary}

\newcommand{\p}[1]{\left(#1\right)}
\newcommand{\st}[1]{\left\{#1\right\}}
\newcommand{\ab}[1]{\left|#1\right|}

\newcommand{\abk}[1]{\left<#1\right>}

\newcommand{\quot}[1]{``#1''}

\DeclareMathOperator{\modd}{mod}

\newcommand{\namehead}[3]{
\lstset{breaklines=true, morecomment=[l]{//}, frame=single, showstringspaces=false, numbers=left}
\begin{flushright}
Nathan Fox\\
#2\\
#3\\
\end{flushright}
\ifstrequal{#1}{.}{}{
\begin{center}
{\Large Homework #1}
\end{center}}
}
\usepackage{multicol}
\usepackage{tikz}
\usepackage{tikz-qtree}

\renewcommand{\p}[1]{(#1)}
\newcommand{\pb}[1]{\left(#1\right)}

\newcommand{\mbb}{\mathbf}

\newcommand{\As}{\lambda}
\newcommand{\Bs}{\mu}

\newcommand{\gh}[1]{\url{http://github.com/nhf216/thesis/blob/master#1}}

\begin{document}
%
%
\title{A New Approach to the Hofstadter $Q$-Recurrence}
\author{Nathan Fox\footnote{Department of Mathematics and Computer Science, The College of Wooster, Wooster, Ohio,
\texttt{nfox@wooster.edu}
}}
\date{}

\maketitle

\begin{abstract}
Nested recurrence relations are highly sensitive to their initial conditions.  The best-known nested recurrence, the Hofstadter $Q$-recurrence, generates sequences displaying a wide variety of behaviors.  Most famous among these is the Hofstadter $Q$-sequence, which appears to be structured at a macro level and chaotic at a micro level.  Other choices of initial conditions can lead to more predictable solutions, frequently interleavings of simple sequences.  Previous work has focused on the form of a desired solution and on describing an initial condition that generates such a solution.  In this paper, we flip this paradigm around.  We illustrate how focusing on the form of an initial condition and describing the resulting sequences can yield strange families of new solutions to nested recurrences.
\end{abstract}

\section{Introduction}
The Hofstadter $Q$-sequence~\cite{geb} is defined by the nested recurrence
\[
Q\p{n}=Q\p{n-Q\p{n-1}}+Q\p{n-Q\p{n-2}}
\]
with the initial conditions $Q\p{1}=1$ and $Q\p{2}=1$.  As successive terms are generated, the $Q$-sequence seems to behave rather chaotically.  But, plots of the sequence suggest that $Q\p{n}$ remains close to $\frac{n}{2}$, and there appears to be a sort of fractal structure in the plot.  See Figure~\ref{fig:hof} for a plot of the first ten thousand terms.

\begin{figure}
\begin{center}
\includegraphics[width=300pt]{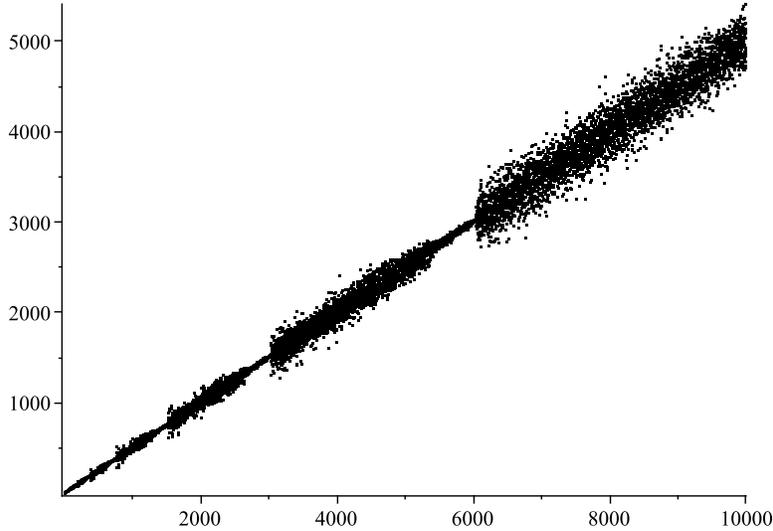}
\end{center}
\caption{Plot of the first $10000$ terms of the Hofstadter $Q$-sequence}
\label{fig:hof}
\end{figure}

A key question regarding the $Q$-sequence is whether it is, in fact, an infinite sequence.  Based on the recurrence, the value of $Q\p{n}$ depends on the value of $Q\p{n-Q\p{n-1}}$.  If $Q\p{n-1}\geq n$, then the value of $Q\p{n}$ depends on the value of $Q$ at a nonpositive index.  Since the $Q$-sequence is only defined starting with $Q\p{1}$, $Q\p{n-Q\p{n-1}}$ would not exist, and hence $Q\p{n}$ would not exist in this scenario.  If a sequence defined by a nested recurrence is finite in this way, we say that the sequence \emph{dies} after $n-1$ terms, or that it dies at index $n$.  It is still open whether the $Q$-sequence dies, but it contains at least $10^{10}$ terms~\cite{oeis}.

\subsection{Notation}

Before we continue, we introduce a few pieces of notation that appear throughout this paper.  The only recurrence relation we discuss is the Hofstadter $Q$-recurrence, but we study it with many different initial conditions.  The notation $Q\p{n}$ refers to the $n$th term of the $Q$-sequence itself.  The notation $Q^*\p{n}$ refers to a generic sequence that satisfies the $Q$-recurrence.  For any other specific sequence satisfying the $Q$-recurrence, we use $Q$ with a subscript that we define for that particular sequence.

We use angle brackets to denote our initial conditions.  For example, $\abk{1,1}$ is shorthand for $Q^*\p{1}=1$ and $Q^*\p{2}=1$, the initial condition for the Hofstadter $Q$-sequence.  Sometimes, it is convenient to define $Q^*\p{n}=0$ for all $n\leq0$, as forcing sequences to die as previously described can limit the diversity of solutions we encounter~\cite{rusk, symbhof}.  This convention is noted with a symbol $\bar{0}$ followed by a semicolon at the start of the initial condition.  For example, $\abk{\bar{0};1,1}$ is shorthand for $Q^*\p{n}=0$ for $n\leq0$, $Q^*\p{1}=1$ and $Q^*\p{2}=1$.

Note that it is still possible for a sequence with such an extended initial condition to be finite.  If $Q^*\p{n-1}=0$ for some $n$, then $Q^*\p{n}$ would be self-referential.  This sort of issue cannot be resolved via an initial condition tweak, so we declare $Q^*\p{n}$ to be undefined in this case.  To avoid confusion with earlier terminology, we do not say that such a sequence dies.  Rather, we say that it \emph{ends} after $n-1$ terms or at index $n$.
%
%

\subsection{Preliminarities}

Nested recurrence relations, such as the Hofstadter $Q$-recurrence, are highly sensitive to their initial conditions.  For example, if we change the initial condition of the $Q$-recurrence to $\abk{3,2,1}$, we obtain a sequence consisting of three interleaved constant or linear sequences~\cite{golomb}, which we denote by $Q_G$:
\[
\begin{cases}
Q_G\p{3k}=3k-2\\
Q_G\p{3k+1}=3\\
Q_G\p{3k+2}=3k+2.
\end{cases}
\]
(To use the above expression to determine $Q_G\p{n}$, first write $n=3k+r$ where $r=n\bmod3$, and then refer to the appropriate case.)  
Going forward, we say that a sequence consisting of $m$ interleaved constant and linear sequences is \emph{quasilinear} with period~$m$.  In this language, $Q_G$ is quasilinear with period~$3$.  Another notable initial condition to the $Q$-recurrence is $\abk{\bar{0};3,6,5,3,6,8}$.  The resulting sequence is not quasilinear; rather, it is an interleaving of two constant sequences with the Fibonacci sequence~\cite{rusk}.


Previous approaches~\cite{golomb, rusk, con, symbhof, genrusk, gengol, slowtrihof, tanny, hofv, erickson, isgur1, isgur2} have focused on trying to find initial conditions to nested recurrences that produce solutions of a specific form.  In this paper, we instead find predictable solutions to the Hofstadter $Q$-recurrence by specifying the form of the initial condition and determining the behavior of the resulting sequence.  In Section~\ref{s:wd}, we characterize the sequences resulting from the family of initial conditions of the form $\abk{1,2,3,\ldots,N}$, and in Section~\ref{s:sd}, we study the more general initial condition $\abk{\bar{0};1,2,3,\ldots,N}$.  Finally, we suggest some future research directions in Section~\ref{s:future}.

\section{A Family of Dying Sequences}\label{s:wd}

In this section, we consider sequences obtained from the Hofstadter $Q$-recurrence and an initial condition of the form $\abk{1,2,3,\ldots,N}$ for some integer $N\geq2$.  Henceforth, we denote this sequence for a given value of $N$ by $Q_N$.
%

We have the following result, which characterizes the behaviors of almost all of these sequences.

\begin{theorem}\label{thm:1thruN}
For $N=8$, $N=11$, $N=12$, or $N\geq14$, the sequence $Q_N$ dies.  Furthermore, if $N\geq21$, the sequence has exactly $N+28$ terms, and if $14\leq N\leq20$, the sequence has exactly $N+32$ terms.
\end{theorem}


\begin{proof}
It is straightforward to verify by computing terms that $Q_8\p{420}=430$, $Q_{11}\p{199}=206$, and $Q_{12}\p{69}=77$, so these sequences all die~\cite{oeis}.

In general, we can compute terms following the initial condition as a function of the parameter $N$.  First, we compute $Q_N\p{N+1}$.  By the $Q$-recurrence, this equals $Q_N\p{N+1-Q_N\p{N}}+Q_N\p{N+1-Q_N\p{N-1}}$.  Both $Q\p{N}$ and $Q\p{N-1}$ lie in the initial condition, so they equal $N$ and $N-1$ respectively.  This allows us to simplify the expression to $Q_N\p{N+1-N}+Q_N\p{N+1-\p{N-1}}=Q_N\p{1}+Q_N\p{2}$.  Again, we have two terms from the initial condition, so we obtain that $Q_N\p{N+1}=3$.  This calculation is invalid if $N=1$ (as then neither index $2$ nor index $N-1$ would be in the initial condition), but it is valid for any $N\geq2$.

Subsequent terms are computed using a similar process.  Two important notes:
\begin{enumerate}[(a)]
\item The terms arising in the intermediate steps are not always from the initial condition.  
But, if a term is not from the initial condition, we can proceed as long as it lies before the current term.  In that eventuality, we would have already computed it, so we can use its computed value.
\item The calculations at each step are only valid for $N$ sufficiently large.  If a fact of the form $Q\p{i}=i$ is used to simplify an expression for some constant $i$, then we must have $N\geq i$.  Similarly, if a fact of the form $Q\p{N-i}=N-i$ is used, we must have $N>i$.
\end{enumerate}

Using this process, we can compute $28$ terms following the initial condition before we run into any issues.  These $28$ terms are:
\begin{align*}
3,\, & N + 1,\, N + 2,\, 5,\, N + 3,\, 6,\, 7,\, N + 4,\, N + 6,\, 10,\, 8,\, N + 6,\, N + 10,\, 12,\, N + 7,\, 14,\, \\ & N + 12,\, 11,\,N + 11,\, N + 15,\, 16,\, 13,\, 17,\, 15,\, N + 14,\, 20,\, 20,\, 2N + 8.
\end{align*}
See Appendix~\ref{app:Qwd} for explicit computations of these terms, along with a bound on the values of $N$ for which that computation and all previous computations are valid.  In particular, note that the calculations are valid for $N\geq13$.

The last term we have is $Q\p{N+28}=2N+8$.  We try to compute $Q\p{N+29}$:
\begin{align*}
    \mbb{Q_N\p{N+29}}&=Q_N\p{N+29-Q_N\p{N+28}}+Q_N\p{N+29-Q_N\p{N+27}}\\
    &=Q_N\p{N+29-\pb{2N+8}}+Q_N\p{N+29-20}\\
    &=Q_N\pb{-N+21}+Q_N\pb{N+9}.
    \end{align*}
If $N\geq21$, then $-N+21\leq0$, so $Q_N\p{-N+21}$ is undefined and the sequence dies.


This just leaves the values $14\leq N\leq20$ to examine.  This is a finite range, so it suffices to individually check that these sequences all die after $N+32$ terms.  But, these seven sequences all die according to the same pattern, so we give a unifying proof for all of them.  Suppose $14\leq N\leq 20$.  We then have $Q_N\p{-N+21}=-N+21$, as that term now lies in the initial condition.  So, we can continue to compute terms.  All calculations below are only valid if $N\geq13$, which is the case for the range we are considering.
    \begin{align*}
    \mbb{Q_N\p{N+29}}&=Q_N\pb{-N+21}+Q_N\pb{N+9}=-N+21+N+6=\mbb{27}
    \end{align*}
    \begin{align*}
    \mbb{Q_N\p{N+30}}&=Q_N\p{N+30-Q_N\p{N+29}}+Q_N\p{N+30-Q_N\p{N+28}}\\
    &=Q_N\p{N+30-27}+Q_N\p{N+30-\pb{2N+8}}\\
    &=Q_N\pb{N+3}+Q_N\pb{-N+22}=N+2-N+22=\mbb{24}
    \end{align*}
    \begin{align*}
    \mbb{Q_N\p{N+31}}&=Q_N\p{N+31-Q_N\p{N+30}}+Q_N\p{N+31-Q_N\p{N+29}}\\
    &=Q_N\p{N+31-24}+Q_N\p{N+31-27}\\
    &=Q_N\pb{N+7}+Q_N\pb{N+4}=7+5=\mbb{12}
    \end{align*}
    \begin{align*}
    \mbb{Q_N\p{N+32}}&=Q_N\p{N+32-Q_N\p{N+31}}+Q_N\p{N+32-Q_N\p{N+30}}\\
    &=Q_N\p{N+32-12}+Q_N\p{N+32-24}\\
    &=Q_N\pb{N+20}+Q_N\pb{N+8}=N+15+N+4=\mbb{2N+19}.
    \end{align*}
If $N\geq14$, then $2N+19\geq N+33$.  This means that, if $14\leq N\leq20$, then $Q_N\p{N+33}$ fails to exist.  So, $Q_N$ dies after $N+32$ terms whenever $14\leq N\leq20$, as required.
\end{proof}

Theorem~\ref{thm:1thruN} says that $Q_N$ dies for all but finitely many $N$.  This begs the question of what happens when $N\in\st{2,3,4,5,6,7,9,10,13}$.  The sequence $Q_2$ is Hofstadter's sequence without the initial $1$, so it is unknown whether $Q_2$ dies.  Since $Q_2\p{3}=3$, $Q_3=Q_2$, so $N=3$ also gives Hofstadter's sequence.  The remaining $N$ values in this set give sequences that are different from Hofstadter's sequence and different from each other.  Like Hofstadter's, it is unknown whether any of these sequences dies.  All of these sequences last for at least $30$ million terms~\cite{oeis}.

\section{More Complicated Behavior}\label{s:sd}
The sequence in Section~\ref{s:wd} almost all die.  Here, we consider what happens if we prevent them from dying by defining their values to be zero at nonpositive integers.  For an integer $N\geq2$, let $Q_{\bar{N}}$ denote the sequence obtained from the Hofstadter $Q$-recurrence with initial condition $\abk{\bar{0};1,2,3,\ldots,N}$.  

%
%
%
Somewhat surprisingly, the behavior of $Q_{\bar N}$ depends on the congruence class of $N$ modulo~$5$.  
Before delving into details, we describe the high-level structure of these sequences for sufficiently large $N$.  First, the sequences $Q_N$ and $Q_{\bar N}$ agree until $Q_N$ dies.  Shortly after that point, $Q_{\bar N}$ settles into a period-$5$ quasilinear pattern.  Unlike the sequence $Q_G$ mentioned in the introduction, where the quasilinear pattern lasts forever, the period-$5$ behavior of $Q_{\bar N}$ is only temporary.  What happens once it collapses depends on $N$ mod~$5$. 
For three congruence classes, a term in $Q_{\bar N}$ depends on itself shortly after the quasilinear behavior stops, causing the sequence to end there.  In one case, the sequence ends only $4$ terms beyond the end of the quasilinear part.  The other two cases last $11$ and $158$ terms beyond it.  Of the two remaining congruence classes, one of them leads to a seemingly infinite sequence some of whose terms are predictable and others of which appear chaotic.  The other class leads to a period-$5$ quasilinear pattern not unlike the one that stopped shortly before.  Like the original period-$5$ pattern, this one is also temporary.  When it finishes, the same five possible continuations of behavior are possible, with the behavior now dependent on $N$ mod $25$.  Similar period-$5$ chunks appear to be possible to arbitrary depths.

The structure of this section is as follows. In~\ref{ss:rst} we formally introduce the semi-predictable sequences discussed above.  Then, in~\ref{ss:main}, we formally state and prove Theorem~\ref{thm:main}, which fully describes the structure of $Q_{\bar N}$.  Then,~\ref{ss:disc} is devoted to a further discussion, in plain language, of the consequences of Theorem~\ref{thm:main}.  
Finally, 
a discussion of the remaining cases, when $N$ is not sufficiently large, is carried out in~\ref{ss:sporadic}.

\subsection{Interlude:\ A Family of Semi-Predictable Solutions}\label{ss:rst}
In order to fully characterize the sequences $Q_{\bar N}$, it is necessary to describe a peculiar family of sequences that satisfy Hofstadter's recurrence.  
Historically, solutions to Hofstadter-like recurrences have looked one of the following:
\begin{enumerate}[(a)]
\item\label{it:die} Finite (dying/ending) sequences (e.g.\ $Q_{70}$)
\item\label{it:hof} Apparently infinite sequences with seemingly chaotic behavior, though perhaps with some detectable patterns (e.g.\ the Hofstadter $Q$-sequence)
\item\label{it:gol} A sequence satisfying a linear recurrence relation (e.g.\ $Q_G$)
\item\label{it:slow} A monotone increasing sequence with successive differences $0$ or $1$ (e.g.\ Tanny's sequence~\cite{tanny})
\end{enumerate}

Now, we describe a family of solutions to the $Q$-recurrence that does not fall cleanly into the above classification, instead combining elements of cases~(\ref{it:hof}) and~(\ref{it:gol}).  In particular, the solutions are interleavings of five sequences.  Four of them are chaotic and seemingly infinite, and the fifth is a constant sequence.


As auxiliary objects, we define three sequences in terms of a system of nested recurrences:
\begin{defin}\label{def:rst}
Define sequences $R\p{n}$, $S\p{n}$, and $T\p{n}$ as follows:
\begin{itemize}
\item $R\p{n}=0$ for $n\leq0$, $R\p{1}=1$, $R\p{2}=2$, $R\p{n}=R\p{n-R\p{n-1}}+S\p{n-1}$ for $n\geq3$
\item $S\p{n}=0$ for $n<0$, $S\p{0}=1$, $S\p{1}=1$, $S\p{n}=S\p{n-R\p{n}}+S\p{n-R\p{n-1}}$ for $n\geq2$
\item $T\p{n}=0$ for $n<0$, $T\p{0}=1$, $T\p{n}=T\p{n-R\p{n}}+T\p{n-S\p{n}}$ for $n\geq1$
\end{itemize}
\end{defin}
Plots of these sequences are given in Figures~\ref{fig:r},~\ref{fig:s}, and~\ref{fig:t} respectively.  All of them appear to behave fairly chaotically, and, much like the $Q$-sequence, it is unknown whether or not they end.  According to the plots, $R$ and $S$ appear to grow approximately linearly, whereas $T$ appears to grow superlinearly.
\begin{figure}
\begin{center}
\includegraphics[width=300pt]{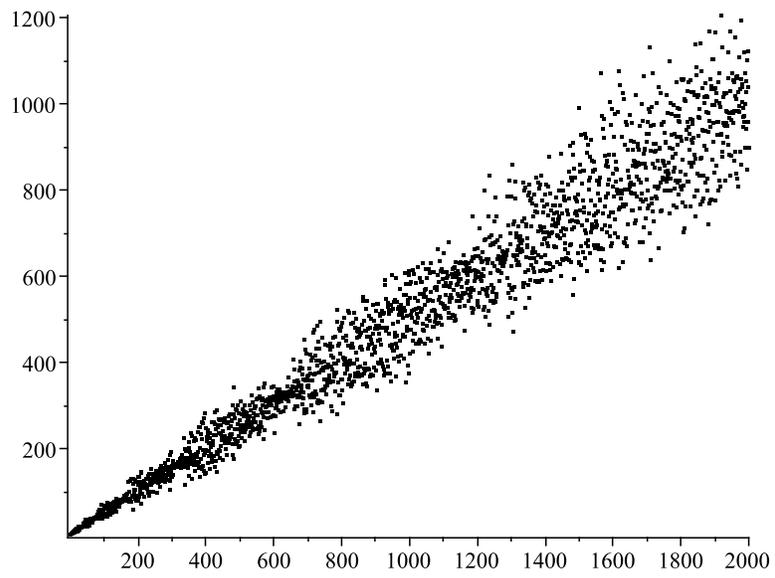}
\end{center}
\caption{Plot of $R\p{1}$ through $R\p{2000}$}
\label{fig:r}
\end{figure}
\begin{figure}
\begin{center}
\includegraphics[width=300pt]{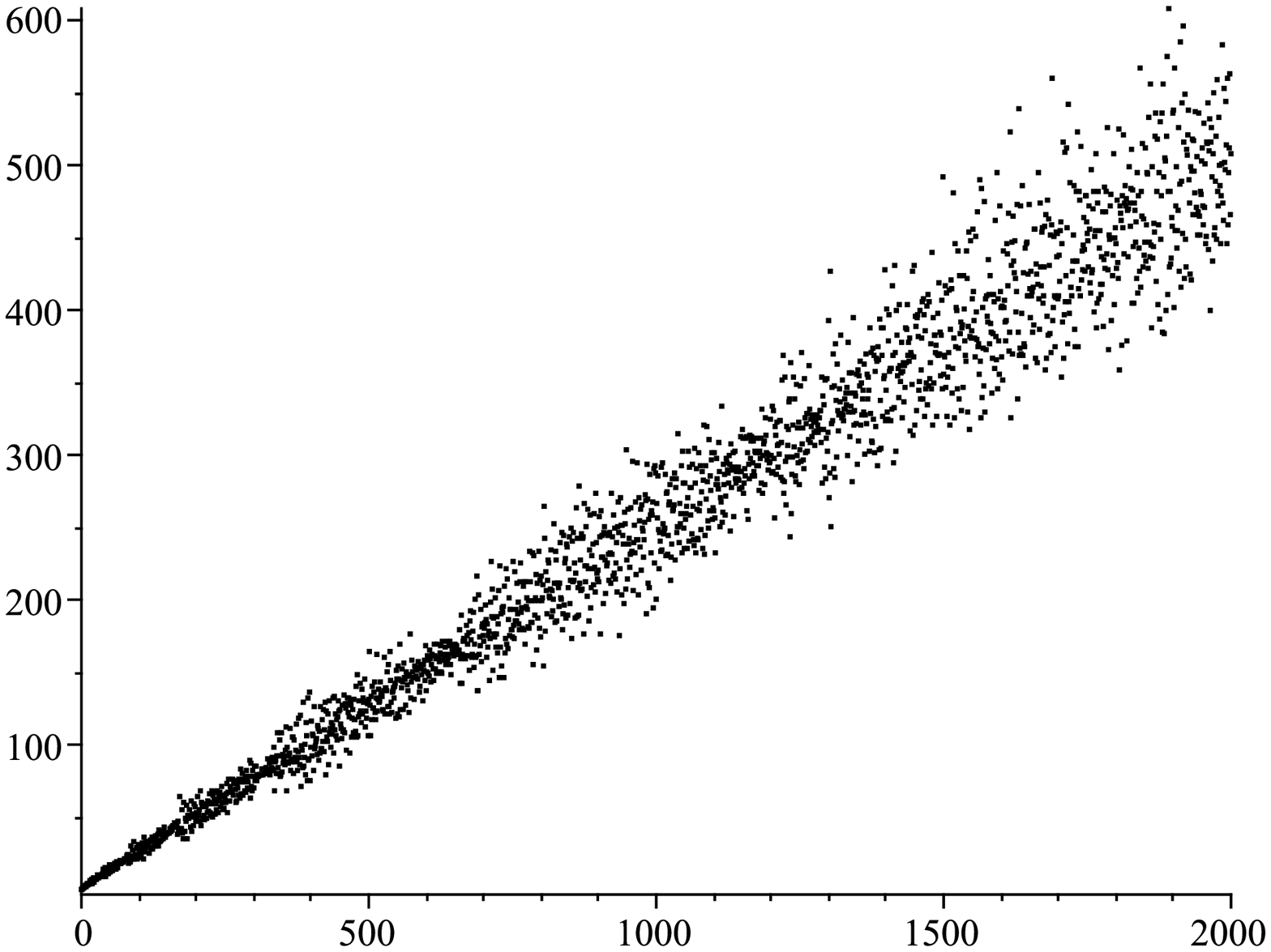}
\end{center}
\caption{Plot of $S\p{0}$ through $S\p{2000}$}
\label{fig:s}
\end{figure}
\begin{figure}
\begin{center}
\includegraphics[width=300pt]{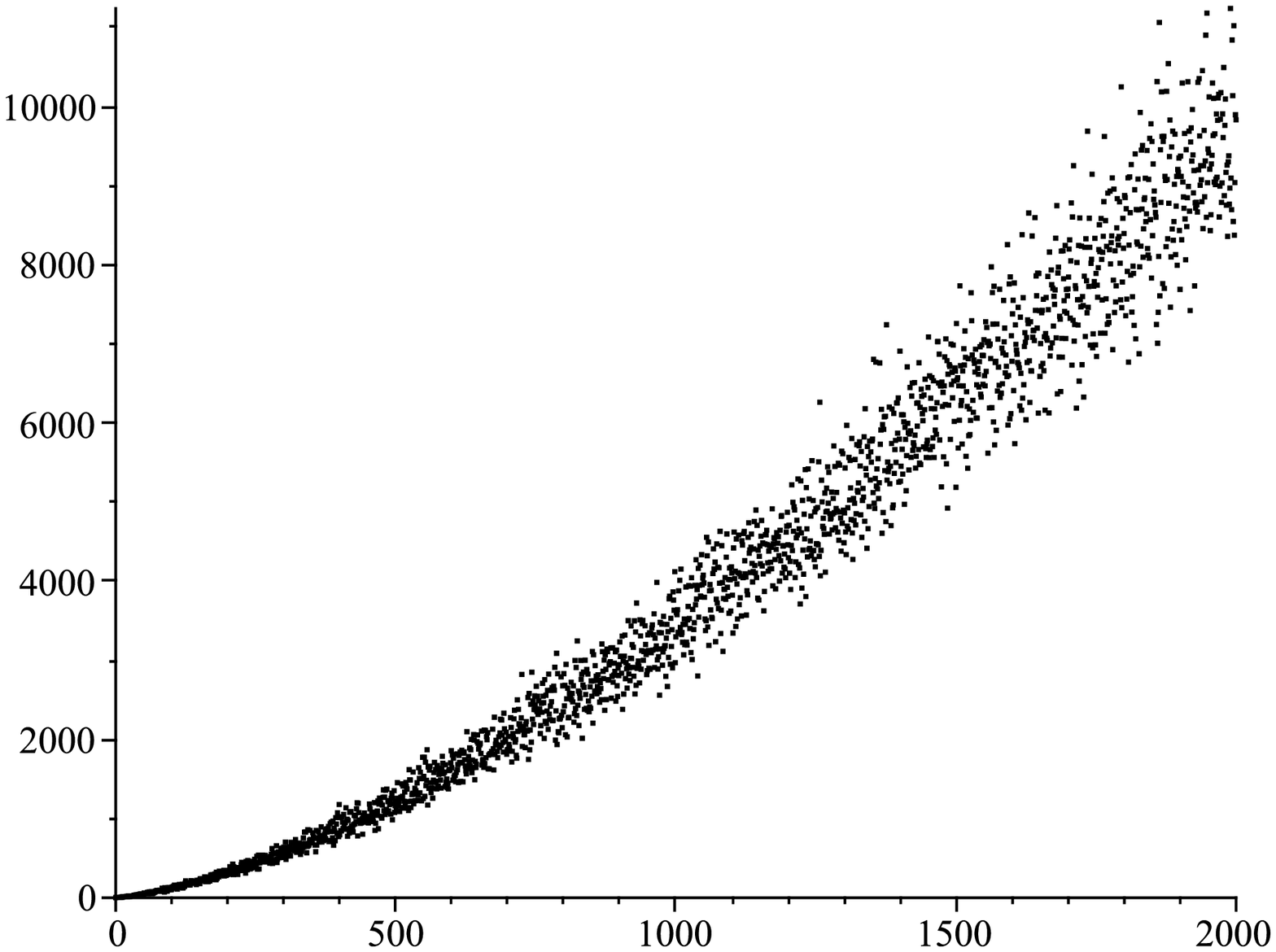}
\end{center}
\caption{Plot of $T\p{0}$ through $T\p{2000}$}
\label{fig:t}
\end{figure}

Multiples of the $R$, $S$, and $T$ sequences can appear as equally-spaced subsequences of solutions to the Hofstadter $Q$-recurrence.  We denote such solutions by $Q_T$.  Proposition~\ref{prop:rst} describes a parametrized family of such solutions.  All solutions in this family eventually consist of five interleaved subsequences: two multiples of $R$, one multiple each of $S$ and $T$, and a sequence of all fours.  In the next section, we will see that some of the sequences $Q_{\bar N}$ are eventually characterized by Proposition~\ref{prop:rst}.
\begin{prop}\label{prop:rst}
Let $K\geq0$, $\As\geq 9$ and $\Bs\geq K+6$ be integers.  The initial condition $\abk{\bar{0};a_1,a_2,\ldots,a_K,5,\As,4,\Bs}$ (each $a_i$ an arbitrary integer) for the Hofstadter $Q$-recurrence generates the following pattern, beginning with index $K+5$ (the first case, with $k=1$),
\[
\begin{cases}
Q_T\p{K+5k}=5R\p{k}\\
Q_T\p{K+5k+1}=5S\p{k}\\
Q_T\p{K+5k+2}=\As T\p{k}\\
Q_T\p{K+5k+3}=4\\
Q_T\p{K+5k+4}=5R\p{k}.
\end{cases}
\]
The pattern lasts as long as the $R$, $S$, and $T$ sequences live and as long as $\As T\p{k}\geq K+5k+4$.
\end{prop}
One may wonder how restrictive the condition $\As T\p{k}\geq K+5k+4$ is. Since the $T$-sequence appears to grow superlinearly, it should be satisfied by sufficiently large $\As$ for fixed $K$.  
In particular, $\As=9$ seems to suffice for $K=0$, lasting for at least fifty million terms (and $\As=8$ fails within the first $60$ terms)~\cite{oeis}. 
The case $K=0$, $\As=9$ and $\Bs=6$ is depicted in Figure~\ref{fig:rst}.

\begin{figure}
\begin{center}
\includegraphics[width=300pt]{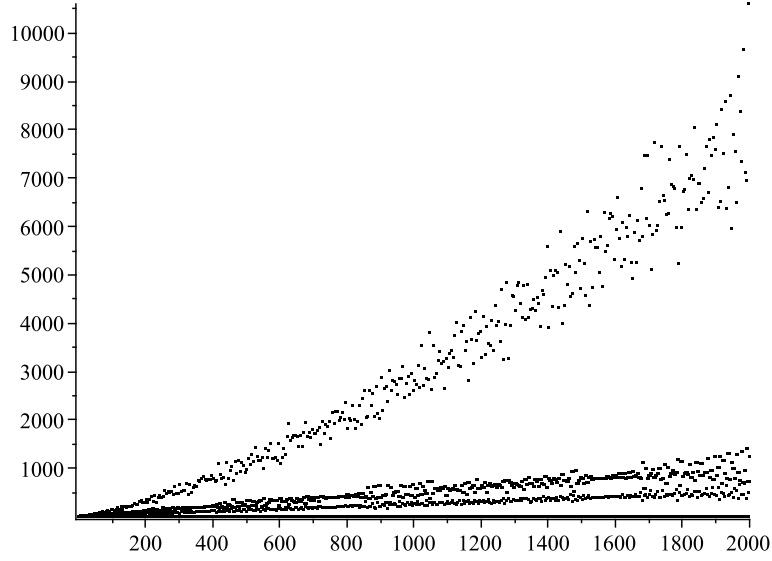}
\end{center}
\caption{The first $2000$ terms of $Q_T$ with initial condition $\abk{\bar{0};5,9,4,6}$}
\label{fig:rst}
\end{figure}

We now prove Proposition~\ref{prop:rst}.
\begin{proof}
The proof is by induction on $n$.  As a base case, we first manually check $n=K+5$ through $n=K+8$.
\begin{itemize}
\item $Q_T\p{K+5}=Q_T\p{K+5-\Bs}+Q_T\p{K+5-4}=Q_T\p{K+1}=5=5R\p{1}$.
\item $Q_T\p{K+6}=Q_T\p{K+6-5}+Q_T\p{K+6-\Bs}=Q_T\p{K+1}=5=5S\p{1}$.
\item $Q_T\p{K+7}=Q_T\p{K+7-5}+Q_T\p{K+7-5}=2Q_T\p{K+2}=2\As=\As T\p{1}$.
\item $Q_T\p{K+8}=Q_T\p{K+8-2\As}+Q_T\p{K+8-5}=Q_T\p{K+3}=4$.
\end{itemize}
(Note that $K+8-2\As\leq0$ because $\As T\p{1}=2\As\geq K+5+4=K+9$.)  We now proceed by induction on $n$ for $n\geq K+9$.  There are $5$ cases to consider.
\begin{description}
\item[$n-K\equiv0\pb{\modd 5}$:] Here, $n=K+5k$ for some $k\geq2$.  We have
\begin{align*}
Q_T\p{K+5k}&=Q_T\p{K+5k-Q_T\p{K+5k-1}}\\
&\hspace{0.17in}+Q_T\p{K+5k-Q_T\p{K+5k-2}}\\
&=Q_T\p{K+5k-5R\p{k-1}}+Q_T\p{K+5k-4}\\
&=5R\p{k-R\p{k-1}}+5S\p{k-1}\\
&=5R\p{k},
\end{align*}
as required.
\item[$n-K\equiv1\pb{\modd 5}$:] Here, $n=K+5k+1$ for some $k\geq2$.  We have
\begin{align*}
Q_T\p{K+5k+1}&=Q_T\p{K+5k+1-Q_T\p{K+5k}}\\
&\hspace{0.17in}+Q_T\p{K+5k+1-Q_T\p{K+5k-1}}\\
&=Q_T\p{K+5k+1-5R\p{k}}+Q_T\p{K+5k+1-5R\p{k-1}}\\
&=5S\p{k-R\p{k}}+5S\p{k-R\p{k-1}}\\
&=5S\p{k},
\end{align*}
as required.
\item[$n-K\equiv2\pb{\modd 5}$:] Here, $n=K+5k+2$ for some $k\geq2$.  We have
\begin{align*}
Q_T\p{K+5k+2}&=Q_T\p{K+5k+2-Q_T\p{K+5k+1}}\\
&\hspace{0.17in}+Q_T\p{K+5k+2-Q_T\p{K+5k}}\\
&=Q_T\p{K+5k+2-5S\p{k}}+Q_T\p{K+5k+2-5R\p{k}}\\
&=\As T\p{k-S\p{k}}+\As T\p{k-R\p{k}}\\
&=\As T\p{k},
\end{align*}
as required.
\item[$n-K\equiv3\pb{\modd 5}$:] Here, $n=K+5k+3$ for some $k\geq2$.  We have
\begin{align*}
Q_T\p{K+5k+3}&=Q_T\p{K+5k+3-Q_T\p{K+5k+2}}\\
&\hspace{0.17in}+Q_T\p{K+5k+3-Q_T\p{K+5k+1}}\\
&=Q_T\p{K+5k+3-\As T\p{k}}+Q_T\p{K+5k+3-5S\p{k}}\\
&=0+4\\
&=4,
\end{align*}
as required.
\item[$n-K\equiv4\pb{\modd 5}$:] Here, $n=K+5k+4$ for some $k\geq1$.  We have
\begin{align*}
Q_T\p{K+5k+4}&=Q_T\p{K+5k+4-Q_T\p{K+5k+3}}\\
&\hspace{0.17in}+Q_T\p{K+5k+4-Q_T\p{K+5k+2}}\\
&=Q_T\p{K+5k+4-4}+Q_T\p{K+5k+4-\As T\p{k}}\\
&=Q_T\p{K+5k}+0\\
&=5R\p{k},
\end{align*}
as required.
\end{description}
What assumptions do we make about $\As$ and $\Bs$?  When computing $Q_T\p{K+6}$, we require $\Bs\geq K+6$.  After this, $\Bs$ never appears again.  For $\As$, when computing $Q_T\p{K+5k+3}$ we need $\As T\p{k}\geq K+5k+4$ for every $k$, as required.
\end{proof}

Aside from the existence of the solutions described in Proposition~\ref{prop:rst} and the single application to the sequences $Q_{\bar N}$, little is know about these and related semi-predictable solutions to the $Q$-recurrence.  A preliminary exploration can be found in~\cite{thesis}.

\subsection{Structure Theorem for $Q_{\bar N}$}\label{ss:main}
In this section, we formally state and prove a theorem (Theorem~\ref{thm:main}) that describes the full behavior of all but finitely many of the sequences $Q_{\bar N}$, modulo open questions about whether or not the sequences in~\ref{ss:rst} are infinite.  The theorem has many parts, all of which have some substance.  Because of the length and amount of technical details in this theorem, we state both a short version and a long version.

Before we state Theorem~\ref{thm:main}, we introduce some auxiliary sequences.
\begin{defin}
Fix an integer $N$.  Define $A_0=N-2$, $A_1=2N+4$,
$
B_1=-11N-22$, and $C_1=\p{N-1}\bmod 5$.  
Then, for $i\geq2$, define
\[
A_{i+1}=A_i\pb{\frac{A_i-A_{i-1}+2}{5}}+B_i,
\]
\[
B_{i+1}=A_{i+1}-A_i,
\]
and
\[
C_i=\p{A_i+2i+1}\bmod 5.
\]
Finally, for all $i\geq1$, define
\[
C'_i=\max\p{0,\p{\p{3-C_i}\modd 5}-1}.
\]
\end{defin}

Note that $A_i$ is not guaranteed to always be an integer, but it is an integer whenever we use it (a fact that is guaranteed by Propopsition~\ref{prop:j5p} on p.~\pageref{prop:j5p}).  We now state our main theorem.
\begin{theorem}[Short Version]
Let $N$ be 
a 
natural number.  
Let $j$ be the first index where $C_j\neq1$ (or $j=\infty$ if $C_j=1$ for all $j$).  Provided $N\geq35$, the sequence $Q_{\bar N}$ has the following structure:
\begin{enumerate}[(a)]
\item For all $1\leq i\leq N$, $Q_{\bar N}\p{i}=i$.
\item The $28$ terms following the initial conditions are the remaining $28$ terms of $Q_N$ (see Appendix~\ref{app:Qwd}).  The sequence then contains six sporadic terms, which are then followed by a quasilinear chunk with period~$5$ that lasts through index $A_1+C'_1$.
\item For each $1\leq m<j$, the previous quasilinear chunk is followed by five sporadic terms and then another quasilinear chunk with period~$5$ that lasts through index $A_{m+1}+C'_{m+1}$.
\item If $C_j=0$ and $N\geq118$, then $Q_{\bar N}$ is finite, and it contains $158$ terms after the last quasilinear chunk concludes.
\item If $C_j=2$, then the behavior of the rest of the sequence is described by Proposition~\ref{prop:rst}, where the initial condition in that proposition is given by the already-generated terms along with the two next terms.
\item If $C_j=3$, then $Q_{\bar N}$ is finite, and it contains $4$ terms after the last quasilinear chunk concludes.
\item If $C_j=4$, then $Q_{\bar N}$ is finite, and it contains $11$ terms after the last quasilinear chunk concludes.
\end{enumerate}
\end{theorem}
\setcounter{theorem}{1}
\begin{theorem}[Full Version]\label{thm:main}
Let $N$ be 
a 
natural number.  
Let $j$ be the first index where $C_j\neq1$ (or $j=\infty$ if $C_j=1$ for all $j$).  Provided $N\geq35$, the sequence $Q_{\bar N}$ has the following structure:
\begin{enumerate}[(a)]
\item\label{it:pa} For all $1\leq i\leq N$, $Q_{\bar N}\p{i}=i$.
\item\label{it:pb} For $1\leq k\leq 28$, $Q_{\bar N}\p{N+k}=Q_N\p{N+k}$ (see Appendix~\ref{app:Qwd}).  The next six terms are $Q_{\bar N}\p{N+29}=N+6$, $Q_{\bar N}\p{N+30}=24$, $Q_{\bar N}\p{N+31}=32$, $Q_{\bar N}\p{N+32}=2N+4$, $Q_{\bar N}\p{N+33}=3$, $Q_{\bar N}\p{N+34}=32$.  Thereafter, for $35\leq 5k+r\leq A_1+C'_1$  with $0\leq r<5$,
\[
\begin{cases}
Q_{\bar N}\p{N+5k}=A_1 k+B_1\\
Q_{\bar N}\p{N+5k+1}=5\\
Q_{\bar N}\p{N+5k+2}=A_1\\
Q_{\bar N}\p{N+5k+3}=3\\
Q_{\bar N}\p{N+5k+4}=5.
\end{cases}
\]
\item\label{it:p1} For each $1\leq m<j$, $Q_{\bar N}\p{A_m+2}=5$, $Q_{\bar N}\p{A_m+3}=8$, $Q_{\bar N}\p{A_m+4}=A_{m+1}$, $Q_{\bar N}\p{A_m+5}=3$, $Q_{\bar N}\p{A_m+6}=8$, and for all $7\leq 5k+r\leq A_{m+1}+C'_{m+1}$ with $0\leq r<5$,
\[
\begin{cases}
Q_{\bar N}\p{A_m+5k}=3\\
Q_{\bar N}\p{A_m+5k+1}=5\\
Q_{\bar N}\p{A_m+5k+2}=A_{m+1}k+B_{m+1}\\
Q_{\bar N}\p{A_m+5k+3}=5\\
Q_{\bar N}\p{A_m+5k+4}=A_{m+1}.
\end{cases}
\]
\item\label{it:p0} If $C_j=0$ and $N\geq118$, then $Q_{\bar N}$ ends after $A_j+160$ terms.  
See Appendix~\ref{app:160} for the remaining $158$ terms.
\item\label{it:p2} If $C_j=2$, then $Q_{\bar N}\p{A_j+1}=4$, $Q_{\bar N}\p{A_j+2}=A_j\pb{\frac{A_j-A_{j-1}-4}{5}}+B_j+2$, and thereafter, for $5k+r\geq3$ with $0\leq r<5$
\[
\begin{cases}
Q_{\bar N}\p{A_j+5k}=A_j T\p{k}\\
Q_{\bar N}\p{A_j+5k+1}=4\\
Q_{\bar N}\p{A_j+5k+2}=5 R\p{k}\\
Q_{\bar N}\p{A_j+5k+3}=5 R\p{k+1}\\
Q_{\bar N}\p{A_j+5k+4}=5 S\p{k+1}
\end{cases}
\]
assuming the $R$, $S$, and $T$ sequences from~\ref{ss:rst} last forever and assuming that, for all $k\geq1$,
\[
T\p{k}\geq1+\frac{5k+2}{A_j}.
\]
\item\label{it:p3} If $C_j=3$, then $Q_{\bar N}$ ends after $A_j+4$ terms.  The remaining $4$ terms are:
\begin{itemize}
\item $Q_{\bar N}\p{A_j+1}=6$
\item $Q_{\bar N}\p{A_j+2}=A_j+5$
\item $Q_{\bar N}\p{A_j+3}=A_j\pb{\frac{A_j-A_{j-1}-5}{5}}+B_j$
\item $Q_{\bar N}\p{A_j+4}=0$
\end{itemize}
\item\label{it:p4} If $C_j=4$, then $Q_{\bar N}$ ends after $A_j+14$ terms.  The remaining $11$ terms are:
\begin{multicols}{2}
\begin{itemize}
\item $Q_{\bar N}\p{A_j+4}=7$
\item $Q_{\bar N}\p{A_j+5}=A_j+5$
\item $Q_{\bar N}\p{A_j+6}=4$
\item $Q_{\bar N}\p{A_j+7}=A_j+2$
\item $Q_{\bar N}\p{A_j+8}=13$
\item $Q_{\bar N}\p{A_j+9}=A_j\pb{\frac{A_j-A_{j-1}-6}{5}}+B_j+7$
\item $Q_{\bar N}\p{A_j+10}=5$
\item $Q_{\bar N}\p{A_j+11}=4$
\item $Q_{\bar N}\p{A_j+12}=A_j+15$
\item $Q_{\bar N}\p{A_j+13}=A_j\pb{\frac{A_j-A_{j-1}-6}{5}}+B_j+7$
\item $Q_{\bar N}\p{A_j+14}=0$
\end{itemize}
\end{multicols}
\end{enumerate}
\end{theorem}

It is worth noting that the condition
\[
T\p{k}\geq1+\frac{5k+2}{A_j}
\]
in part~(\ref{it:p2}) of Theorem~\ref{thm:main} is almost certainly not necessary, since the $T$-sequence appears to grow superlinearly and $A_j$ is always at least $80$ (and often much larger).

The proof of Theorem~\ref{thm:main} requires the following lemma, which is of a similar flavor to Proposition~\ref{prop:rst}.
\begin{lemma}\label{lem:5cyc}
Let $K\geq0$ be an integer, and let $\As$ and $\Bs$ be any integers satisfying $\As>K+5$ and 
$\As+\Bs>K+6$.  Let $\nu=\max\p{0,\p{\p{K+4-\As}\modd5}-1}$.  
Then, for arbitrary integers $a_1,a_2,\ldots,a_K$, denote the sequence resulting from the Hofstadter $Q$-recurrence and the 
initial condition $\abk{\bar{0};a_1,a_2,\ldots,a_K,\Bs,5,\As,3}$ by $Q_C$.  
The sequence $Q_C$ follows the following pattern from $Q_C\p{K+1}$ through $Q_C\p{\As+\nu}$:
\[
\begin{cases}
Q_C\p{K+5k}=5\\
Q_C\p{K+5k+1}=\As k+\Bs\\
Q_C\p{K+5k+2}=5\\
Q_C\p{K+5k+3}=\As\\
Q_C\p{K+5k+4}=3.
\end{cases}
\]
%
\end{lemma}
\begin{proof}
%
The proof is by induction on the index.  The base cases are $Q_C\p{K+1}$ through $Q_C\p{K+4}$, which are part of the initial condition.  Now, suppose $K+5\leq n\leq \As$ (we handle indices greater than $\As$ later), and suppose that $Q_C\p{n'}$ is what we want it to be for all $K+1\leq n'<n$.  
There are five cases to consider:
\begin{description}
\item[$n-K\equiv0\pb{\modd 5}$:] In this case, $n=K+5k$ for some $k$.  Applying the $Q$-recurrence, we have
\begin{align*}
Q_C\p{K+5k}&=Q_C\p{K+5k-Q_C\p{K+5k+4}}\\
&\hspace{0.17in}+Q_C\p{K+5k-Q_C\p{K+5k+3}}\\
&=Q_C\p{K+5k-3}+Q_C\p{K+5k-\As}\\
&=5+0\\
&=5,
\end{align*}
as required.  Note that the validity of this case depends on $n\leq\lambda$.
\item[$n-K\equiv1\pb{\modd 5}$:] In this case, $n=5k+1$ for some $k$.  Applying the $Q$-recurrence, we have
\begin{align*}
Q_C\p{K+5k+1}&=Q_C\p{K+5k+1-Q_C\p{K+5k}}\\
&\hspace{0.17in}+Q_C\p{K+5k+1-Q_C\p{K+5k-1}}\\
&=Q_C\p{K+5k+1-5}+Q_C\p{K+5k+1-3}\\
&=\As\p{k-1}+\Bs+\As\\
&=\As k+\Bs,
\end{align*}
as required.  Note that the validity of this case does not depend on $\As$ or $\Bs$.
\item[$n-K\equiv2\pb{\modd 5}$:] In this case, $n=K+5k+2$ for some $k$.  Applying the $Q$-recurrence, we have
\begin{align*}
Q_C\p{K+5k+2}&=Q_C\p{K+5k+2-Q_C\p{K+5k+1}}\\
&\hspace{0.17in}+Q_C\p{K+5k+2-Q_C\p{K+5k}}\\
&=Q_C\p{K+5k+2-\p{\As k+\Bs}}+Q_C\p{K+5k+2-5}\\
&=0+5\\
&=5,
\end{align*}
as required.  Note that the validity of this case does not depend on $\As$ or $\Bs$.
\item[$n-K\equiv3\pb{\modd 5}$:] In this case, $n=5k+3$ for some $k$.  Applying the $Q$-recurrence, we have
\begin{align*}
Q_C\p{K+5k+3}&=Q_C\p{K+5k+3-Q_C\p{K+5k+2}}\\
&\hspace{0.17in}+Q_C\p{K+5k+3-Q_C\p{K+5k+1}}\\
&=Q_C\p{K+5k+3-5}+Q_C\p{K+5k+3-\p{\As k+\Bs}}\\
&=\As+0\\
&=\As,
\end{align*}
as required.  Note that the validity of this case does not depend on $\As$ or $\Bs$.
\item[$n-K\equiv4\pb{\modd 5}$:] In this case, $n=5k+4$ for some $k$.  Applying the $Q$-recurrence, we have
\begin{align*}
Q_C\p{5k+4}&=Q_C\p{K+5k+4-Q_C\p{K+5k+3}}\\
&\hspace{0.17in}+Q_C\p{K+5k+4-Q_C\p{K+5k+2}}\\
&=Q_C\p{K+5k+4-\As}+Q_C\p{K+5k+4-5}\\
&=0+3\\
&=3,
\end{align*}
as required.  Note that the validity of this case depends on $n\leq\lambda$.
\end{description}

This proves that the pattern lasts through index $\As$.  We now complete the proof by showing that the pattern continues through index $\As+\nu$.  There are five cases to consider:
\begin{description}
\item[$\As-K\equiv0\pmod5$:] In this case, $\nu=3$, and the calculation of $Q_C\p{\As}$ falls into the first case above.  The next three cases do not depend on $\As$, so the values of $Q_C\p{\As+1}$, $Q_C\p{\As+2}$, and $Q_C\p{\As+3}$ are what we want.
\item[$\As-K\equiv1\pmod5$:] In this case, $\nu=2$, and the calculation of $Q_C\p{\As}$ falls into the second case above.  The next two cases do not depend on $\As$, so the values of $Q_C\p{\As+1}$ and $Q_C\p{\As+2}$ are what we want.
\item[$\As-K\equiv2\pmod5$:] In this case, $\nu=1$, and the calculation of $Q_C\p{\As}$ falls into the third case above.  The next case does not depend on $\As$, so the value of $Q_C\p{\As+1}$ is what we want.
\item[$\As-K\equiv3\pmod5$:] In this case, $\nu=0$, so there is nothing to be checked.
\item[$\As-K\equiv4\pmod5$:] In this case, $\nu=0$, so there is nothing to be checked.
\end{description}
\end{proof}

We now prove Theorem~\ref{thm:main}.
\begin{proof}

We refer the reader to Appendix~\ref{app:Qwd} for terms $Q_{\bar N}\p{1}$ through $Q_{\bar N}\p{N+28}$.  Those calculations, which are for $Q_N$, also apply for $Q_{\bar N}$.  From there, it is easy to compute $Q_{\bar N}\p{N+29}$ through $Q_{\bar N}\p{N+34}$, and each one equals its purported value.  We now compute the next four terms:
\begin{itemize}
\item $Q_{\bar N}\p{N+35}=Q_{\bar N}\p{N+3}+Q_{\bar N}\p{N+32}=\pb{N+2}+\pb{2N+4}=3N+6$.
\item $Q_{\bar N}\p{N+36}=Q_{\bar N}\p{N+4}=5$.
\item $Q_{\bar N}\p{N+37}=Q_{\bar N}\p{N+32}=2N+4=A_1$.
\item $Q_{\bar N}\p{N+38}=Q_{\bar N}\p{N+33}=3$.
\end{itemize}
By Lemma~\ref{lem:5cyc}, taking $K=N+34$, $\As=2N+4$, and $\Bs=3N+6$, these four terms spawn a period-$5$ pattern:
\[
\begin{cases}
Q_{\bar N}\p{N+34+5k}=5\\
Q_{\bar N}\p{N+34+5k+1}=\pb{2N+4}k+\pb{3N+6}\\
Q_{\bar N}\p{N+34+5k+2}=5\\
Q_{\bar N}\p{N+34+5k+3}=2N+4\\
Q_{\bar N}\p{N+34+5k+4}=3,
\end{cases}
\]
provided that $N>35$.  Lemma~\ref{lem:5cyc} then guarantees that this pattern persists through 
index $Q_{\bar N}\p{A_1+\nu}$, where
\begin{align*}
\nu&=\max\p{0,\p{\p{N+34+4-A_1}\modd5}-1}\\
&=\max\p{0,\p{\p{N+34+4-2N-4}\modd5}-1}\\
&=\max\p{0,\p{\p{34-N}\modd5}-1}\\
&=\max\p{0,\p{\p{4-N}\modd5}-1}\\
&=\max\p{0,\p{\p{3-\p{N-1}}\modd5}-1}\\
&=\max\p{0,\p{\p{3-C_1}\modd5}-1}\\
&=C'_1,
\end{align*}
as required. 
Shifting indices and recalling the definitions of $A_1$ and $B_1$ allows us to rewrite this pattern as
\[
\begin{cases}
Q_{\bar N}\p{N+5k}=A_1k+B_1\\
Q_{\bar N}\p{N+5k+1}=5\\
Q_{\bar N}\p{N+5k+2}=A_1\\
Q_{\bar N}\p{N+5k+3}=3\\
Q_{\bar N}\p{N+5k+4}=5,
\end{cases}
\]
which is the required form.
We now prove part~(\ref{it:p1}) of Theorem~\ref{thm:main}, which refers to a parameter $1\leq m<j$.  Suppose inductively that we are considering the value $m<j$, and that $Q_{\bar N}\p{A_m-3}=3$, $Q_{\bar N}\p{A_m-2}=5$, $Q_{\bar N}\p{A_m-1}=A_m\pb{\frac{A_m-A_{m-1}-3}{5}}+B_m$, $Q_{\bar N}\p{A_m}=5$, and $Q_{\bar N}\p{A_m+1}=A_m$. Note that this is all true if $m=0$, from the above.  So, $m=0$ serves as our (already proved) base case.


Since $m<j$, it must be the case that $C'_m=1$ (as $C_m=1$ implies $C'_m=1$). So, $Q_{\bar N}\p{A_m+2}$ is the first non-calculated term. 
We compute the next $9$ terms:
\begin{itemize}
\item $Q_{\bar N}\p{A_m+2}=Q_{\bar N}\p{2}+Q_{\bar N}\p{A_m-3}=2+3=5$.
\item $Q_{\bar N}\p{A_m+3}=Q_{\bar N}\p{A_m-2}+Q_{\bar N}\p{3}=5+3=8$.
\item $Q_{\bar N}\p{A_m+4}=Q_{\bar N}\p{A_m-4}+Q_{\bar N}\p{A_m-1}$.  We have that $Q_{\bar N}\p{A_m-4}=A_m$.  But, 
$Q_{\bar N}\p{A_m-1}=A_m\pb{\frac{A_m-A_{m-1}-3}{5}}+B_m$. 
So,
\begin{align*}
Q_{\bar N}\p{A_m+4}&=A_m\pb{1+\frac{A_m-A_{m-1}-3}{5}}+B_m\\
&=A_m\pb{\frac{A_m-A_{m-1}+2}{5}}+B_m\\
&=A_{m+1}.
\end{align*}
This term is much larger than $A_m$.
\item $Q_{\bar N}\p{A_m+5}=Q_{\bar N}\p{A_m-3}=3$.
\item $Q_{\bar N}\p{A_m+6}=Q_{\bar N}\p{A_m+1}=8$.
\item $Q_{\bar N}\p{A_m+7}=Q_{\bar N}\p{A_m-1}+Q_{\bar N}\p{A_m+4}$.  We have from before 
$Q_{\bar N}\p{A_m-1}=A_m\pb{\frac{A_m-A_{m-1}-3}{5}}+B_m$. 
But, our calculations in the $Q_{\bar N}\p{A_m+4}$ step allow us to write $Q_{\bar N}\p{A_m-1}=A_{m+1}-A_m$.  So,
$Q_{\bar N}\p{A_m+7}=A_{m+1}-A_m+A_{m+1}=2A_{m+1}-A_m=A_{m+1}+B_{m+1}$.
\item $Q_{\bar N}\p{A_m+8}=Q_{\bar N}\p{A_m}=5$.
\item $Q_{\bar N}\p{A_m+9}=Q_{\bar N}\p{A_m+4}=A_{m+1}$.
\item $Q_{\bar N}\p{A_m+10}=Q_{\bar N}\p{A_m+5}=3$.
\end{itemize}
The first five of these terms are what we want.  And, by Lemma~\ref{lem:5cyc}, the last four terms generate a period-$5$ pattern as in the lemma statement, with $K=A_m+6$, $\As=A_{m+1}$, and $\Bs=A_{m+1}+B_{m+1}$.  
The resulting pattern is
\[
\begin{cases}
Q_{\bar N}\p{A_m+6+5k}=5\\
Q_{\bar N}\p{A_m+6+5k+1}=A_{m+1}\p{k+1}+B_{m+1}\\
Q_{\bar N}\p{A_m+6+5k+2}=5\\
Q_{\bar N}\p{A_m+6+5k+3}=A_{m+1}\\
Q_{\bar N}\p{A_m+6+5k+4}=3,
\end{cases}
\]
which lasts through index $A_{m+1}+\nu$, where
\begin{align*}
\nu&=\max\p{0,\p{\p{A_m+6+4-A_{m+1}}\modd5}-1}\\
&=\max\p{0,\p{\p{A_m-A_{m+1}}\modd5}-1}.
\end{align*}
Shifting indices by $6$, the pattern can be rewritten as
\[
\begin{cases}
Q_{\bar N}\p{A_m+5k}=3\\
Q_{\bar N}\p{A_m+5k+1}=5\\
Q_{\bar N}\p{A_m+5k+2}=A_{m+1}k+A_{m+1}-A_m=A_{m+1}k+B_{m+1}\\
Q_{\bar N}\p{A_m+5k+3}=5\\
Q_{\bar N}\p{A_m+5k+4}=A_{m+1},
\end{cases}
\]
the required form

To complete the proof of part~(\ref{it:p1}) of the theorem, we need to show that $\nu=C_{m+1}'$.  We know that $C_{m+1}\equiv\pb{A_{m+1}+2m+3}\bmod 5$.  This means that $A_{m+1}\equiv \pb{C_{m+1}-2m-3}\bmod 5$.  Similarly, $A_m\equiv \pb{C_m-2m-1}\bmod 5$.  But, we know that $C_m=1$.  So, $A_m\equiv-2m\bmod 5$.  Combining these yields $A_{m+1}-A_m\equiv \pb{C_{m+1}-3}\modd5$.  This allows us to say that
\begin{align*}
\nu&=\max\p{0,\p{\p{A_m-A_{m+1}}\modd5}-1}\\
&=\max\p{0,\p{\p{3-C_{m+1}}\modd5}-1}\\
&=C'_{m+1},
\end{align*}
as required.
All that remains now is to determine the eventual behaviors for $C_j\in\st{0,2,3,4}$ (parts~(\ref{it:p0}), (\ref{it:p2}), (\ref{it:p3}), and~(\ref{it:p4}) of the theorem respectively).
\begin{description}
\item[$C_j=0$:] The first term here we have not yet computed is $Q_{\bar N}\p{A_j+3}$.  We compute the next $158$ terms (see Appendix~\ref{app:160}), and we observe that the sequence ends once $Q_{\bar N}\p{A_j+160}=0$.  Computation of these terms assumes that $N\geq118$, because computing $Q_{\bar N}\p{A_j+157}$ refers to $Q_{\bar N}\p{118}$, which we assume equals $118$ (and this is the strongest requirement we use anywhere in the calculations).
\item[$C_j=2$:] The first term here we have not yet computed is $Q_{\bar N}\p{A_j+1}$.  We compute the next $2$ terms (keeping in mind that $Q_{\bar N}\p{A_j}=A_j$ and $Q_{\bar N}\p{A_j-1}=5$):
\begin{itemize}
\item $Q_{\bar N}\p{A_j+1}=Q_{\bar N}\p{1}+Q_{\bar N}\p{A_j-4}=1+3=4$.
\item $Q_{\bar N}\p{A_j+2}=Q_{\bar N}\p{A_j-2}+Q_{\bar N}\p{2}=A_j\pb{\frac{A_j-A_{j-1}-4}{5}}+B_j+2$.
\end{itemize}
We now have the sort of initial condition described by Proposition~\ref{prop:rst} with $K=A_j-2$, $\As=A_j$, and $\Bs=A_j\pb{\frac{A_j-A_{j-1}-4}{5}}+B_j+2$.  By Proposition~\ref{prop:rst}, this results in the pattern
\[
\begin{cases}
Q_{\bar N}\p{A_j-2+5k}=5 R\p{k}\\
Q_{\bar N}\p{A_j-2+5k+1}=5 S\p{k}\\
Q_{\bar N}\p{A_j-2+5k+2}=A_j T\p{k}\\
Q_{\bar N}\p{A_j-2+5k+3}=4\\
Q_{\bar N}\p{A_j-2+5k+4}=5 R\p{k},
\end{cases}
\]
as long as the $R$, $S$, and $T$ sequences exist and as long as $A_j T\p{k}\geq A_j-2+5k+4$.  This last condition is equivalent to
\[
T\p{k}\geq1+\frac{5k+2}{A_j}.
\]
Shifting indices by $2$, the pattern can be rewritten as
\[
\begin{cases}
Q_{\bar N}\p{A_j+5k}=A_j T\p{k}\\
Q_{\bar N}\p{A_j+5k+1}=4\\
Q_{\bar N}\p{A_j+5k+2}=5 R\p{k}\\
Q_{\bar N}\p{A_j+5k+3}=5 R\p{k+1}\\
Q_{\bar N}\p{A_j+5k+4}=5 S\p{k+1},
\end{cases}
\]
as required.  
\item[$C_j=3$:] The first term here we have not yet computed is $Q_{\bar N}\p{A_j+1}$.  We compute the next $4$ terms, obtaining the values in the theorem statement. We observe that the sequence ends once $Q_{\bar N}\p{A_j+4}=0$.
\item[$C_j=4$:] The first term here we have not yet computed is $Q_{\bar N}\p{A_j+4}$.  We compute the next $11$ terms, obtaining the values in the theorem statement. We observe that the sequence ends once $Q_{\bar N}\p{A_j+14}=0$.
\end{description}
\end{proof}

\subsection{Discussion of Theorem~\ref{thm:main}}\label{ss:disc}

See Figure~\ref{fig:42} for a plot of the first $30000$ terms of $Q_{\overline{42}}$.  For $N=42$, we have $j=3$ and $C_3=2$, so, after the initial condition, there is the zone before $Q_{42}$ dies, followed by a (very short) quasilinear piece, followed by two (successively longer) quasilinear pieces, followed by the eventual Proposition~\ref{prop:rst}-like behavior.  Both axes have logarithmic scales, as otherwise the third quasilinear piece would dominate the plot.  (Each $A_i$ is on the order of the square of the previous one.)

\begin{figure}
\begin{center}
\includegraphics[width=300pt]{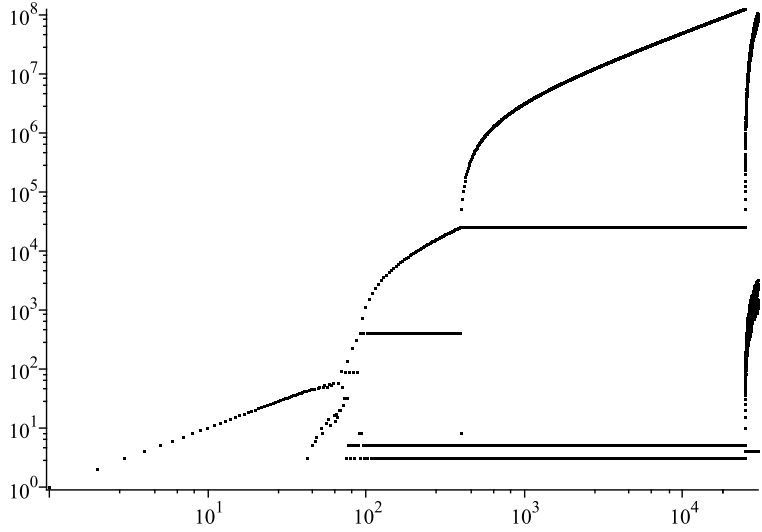}
\end{center}
\caption{The first $30000$ terms of $Q_{\overline{42}}$ (both axes log scale)}
\label{fig:42}
\label{o:A274055}
\end{figure}

Theorem~\ref{thm:main} completely characterizes the behavior of $Q_{\bar N}$ (as long as $N$ is sufficiently large and as long as conjectures about the $R$, $S$, and $T$ sequences hold), but the characterization of which $N$ result in which behavior is not immediately apparent.  Every $N$ with $j<\infty$ (which is every known value of $N$) is associated to a pair $\pb{j,C_j}\in\mathbb{Z}_{>0}\times\st{0,2,3,4}$.  
We denote these values by $j\p{N}$ and $C\p{N}$ respectively.  We also use notation $A_i\p{N}$, $B_i\p{N}$, and $C_i\p{N}$ to denote $A_i$, $B_i$, and $C_i$ values for $N$.  Our first observation is the following:
\begin{prop}\label{prop:j5p}
Let $N$ be a positive integer, and let $j=j\p{N}$.  For all $1\leq i\leq j$, $A_i\p{N+5^j}\equiv A_i\p{N}\pb{\modd 5^{j-i+1}}$.
\end{prop}
\begin{proof}
The proof is by induction on $i$.  If $i=1$, then $A_1\p{N}=2N+4$ and $A_1\p{N+5^j}=2\pb{N+5^j}+4=2N+4+2\cdot5^j$.  Then, $A_1\p{N+5^j}-A_1\p{N}=2\cdot5^j$, which is divisible by $5^j=5^{j-1+1}$, as required.  If $i=2$, then 
\[
A_2\p{N}=\frac{2}{5}N^2-7N-\frac{78}{5}
\]
and
\[
A_2\p{N+5^j}=\frac{2}{5}N^2-7N-\frac{78}{5}+2\cdot5^{2j-1}-7\cdot 5^j+4\cdot 5^{j-1}.
\]
The difference is divisible by $5^{j-1}$, as required.

Now, suppose $i\geq3$ and suppose that Proposition~\ref{prop:j5p} holds for all smaller $i$ values.  Recall that
\[
A_i=A_{i-1}\pb{\frac{A_{i-1}-A_{i-2}+2}{5}}+B_{i-1}.
\]
Since $i\geq3$, $B_{i-1}=A_{i-1}-A_{i-2}$, so we can eliminate $B_{i-1}$ and write
\[
A_i=A_{i-1}\pb{\frac{A_{i-1}-A_{i-2}+7}{5}}-A_{i-2}.
\]
By induction, $A_{i-1}\p{N+5^j}=A_{i-1}\p{N}+\alpha\cdot5^{j-i+2}$ for some integer $\alpha$.  Similarly, $A_{i-2}\p{N+5^j}=A_{i-2}\p{N}+\beta\cdot5^{j-i+3}$ for some integer $\beta$.

We now evaluate
\begin{align*}
A_i\p{N+5^j}-A_i\p{N}&=A_{i-1}\p{N+5^j}\pb{\frac{A_{i-1}\p{N+5^j}-A_{i-2}\p{N+5^j}+7}{5}}\\
&\hspace{0.17in}-A_{i-2}\p{N+5^j}-A_{i-1}\p{N}\pb{\frac{A_{i-1}\p{N}-A_{i-2}\p{N}+7}{5}}\\
&\hspace{0.17in}-A_{i-2}\p{N}\\
&=\pb{A_{i-1}\p{N}+\alpha\cdot5^{j-i+2}}\\
&\hspace{0.17in}\cdot\pb{\frac{\pb{A_{i-1}\p{N}+\alpha\cdot5^{j-i+2}}-\pb{A_{i-2}\p{N}+\beta\cdot5^{j-i+3}}+7}{5}}\\
&\hspace{0.17in}-\pb{A_{i-2}\p{N}+\beta\cdot5^{j-i+2}}\\
&\hspace{0.17in}-A_{i-1}\p{N}\pb{\frac{A_{i-1}\p{N}-A_{i-2}\p{N}+7}{5}}-A_{i-2}\p{N}.
\end{align*}
Simplifying this expression yields
\begin{align*}
A_i\p{N+5^j}-A_i\p{N}&=5^{j-i+1}\pb{2\alpha A_{i-1}\p{N}-5\beta A_{i-1}\p{N}-\alpha A_{i-2}\p{N}\right.\\
&\hspace{0.2in}\left.+\alpha^2\cdot5^{j-i+2}-\alpha\cdot5^{j-i+3}+7\alpha-25\beta},
\end{align*}
which is divisible by $5^{j-i+1}$, as required.
\end{proof}
Of course, Proposition~\ref{prop:j5p} immediately generalizes to replacing $5^j$ with any integer multiple of $5^j$.  We have the following corollary to Proposition~\ref{prop:j5p} (which also generalizes in this way):
\begin{cor}\label{cor:j5p}
For all $N$, and for all $1\leq i\leq j\p{N}$, $C_i\!\pb{N+5^{j\p{N}}}=C_i\p{N}$.  In particular, $j\!\pb{N+5^{j\p{N}}}=j\p{N}$.
\end{cor}
\begin{proof}
Let $j=j\p{N}$.  Let $1\leq i\leq j$.  By Proposition~\ref{prop:j5p},
\[
A_i\p{N+5^j}\equiv A_i\p{N}\pb{\modd 5^{j-i+1}}.
\]
Since $C_i$ is a function solely of $A_i \bmod 5$ and of $i$, we have $C_i\p{N+5^j}=C_i\p{N}$.  Since $i$ is arbitrary in the preceding expression, we have $C_i\p{N+5^j}=C_i\p{N}$ for every such $i$, as required.  Also, $C_i\p{N}=1$ if $i<j$ (by the definition of $j$).  So, by the definition of $j$, we have $j\pb{N+5^j}=j\p{N}$, as required.
\end{proof}

Corollary~\ref{cor:j5p} tells us that, to determine the behavior of $Q_{\bar N}$, we should first look at $N\bmod 5$.  If $C_1\p{N}=1$, then we need to look at $N\bmod 25$.  If  $C_2\p{N}=1$, then we need to look at $N\bmod 125$, etc.  
This process can be thought of in terms of a tree on a subset of the strings $\st{0,1,2,3,4}^*$, each of which can be thought of as an integer written in base~$5$.  In addition, each leaf of the tree has one of four \quot{types.}
\begin{itemize}
\item The root of the tree is the empty string, and it has the five length-$1$ strings as children.
\item For a string $w$, interpret it as a base $5$ integer $N_w$.  Let $C=C_{\ab{w}}\p{N_w}$ (where $\ab{w}$ denotes the length of $w$).  If $C=1$, then $w$ has children $\st{xw:x\in\st{0,1,2,3,4}}$; otherwise $w$ is a leaf of type $C$.
\end{itemize}
To determine the behavior of $Q_{\bar N}$, read the base-$5$ digits of $N$ from right to left, and traverse the tree accordingly.  Each internal node (non-leaf) visited corresponds to an additional temporary quasilinear piece in the sequence.  When a leaf is reached, stop, and the leaf's type determines the eventual behavior, according to Theorem~\ref{thm:main}.  Consider $N=42$ as an example.  In base~$5$, $42$ is $132_5$.  The last digit is $2$, so we go from the root of the tree to the node labeled $2$.  This is not a leaf, so we go from it to the node labeled $32$.  This is also not a leaf, so we continue to the node labeled $132$.  This is a leaf of type $2$.  So, $Q_{\overline{42}}$ consists of three period-$5$ quasilinear pieces followed by a Proposition~\ref{prop:rst}-like piece.

The tree has a structure consisting of levels: level $i$ consists of the strings of length $i$ that appear in the tree.  See Figure~\ref{fig:tree1} for a diagram of levels $0$ through $7$ of the tree.  Each leaf is labeled as its type ($0$, $2$, $3$, or $4$).  The black nodes are internal nodes.  This includes the black nodes on the right.  Each of these has five children, but they are not shown because the tree is truncated.  The structure of this tree is poorly understood.  See~\cite{thesis} for further discussion.

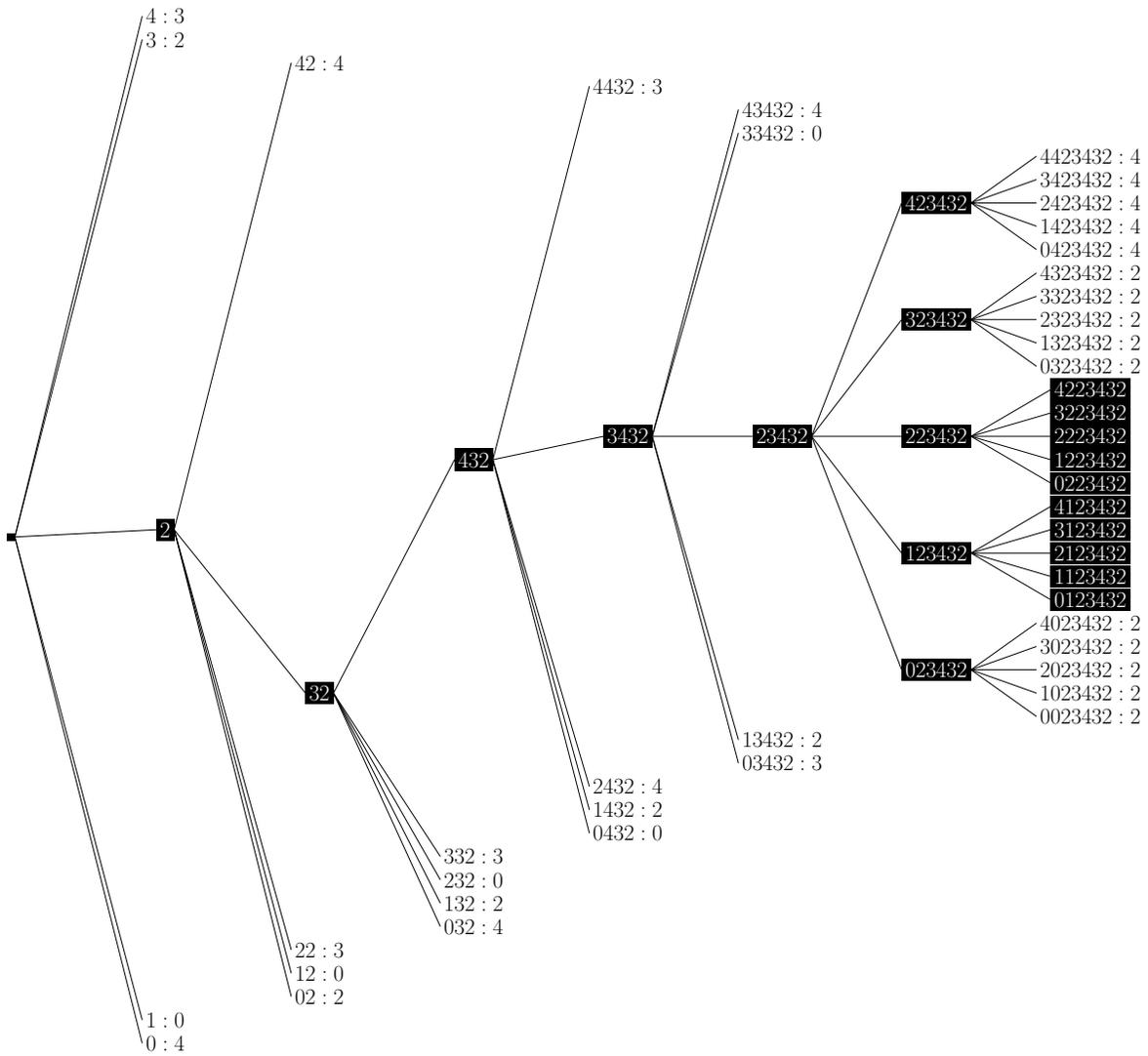
\begin{figure}
\begin{center}
\tikzstyle{c0}=[]
\tikzstyle{c1}=[color=white, top color=black, bottom color = black]
\tikzstyle{c2}=[]
\tikzstyle{c3}=[]
\tikzstyle{c4}=[]
\begin{tikzpicture}[level distance = 6cm, grow=right,scale=0.35]
\Tree [.\node[c1]{{\Huge $$}};
    \node[c4]{{\Huge $0:4$}};
    \node[c0]{{\Huge $1:0$}};
    [.\node[c1]{{\Huge $2$}};
        \node[c2]{{\Huge $02:2$}};
        \node[c0]{{\Huge $12:0$}};
        \node[c3]{{\Huge $22:3$}};
        [.\node[c1]{{\Huge $32$}};
            \node[c4]{{\Huge $032:4$}};
            \node[c2]{{\Huge $132:2$}};
            \node[c0]{{\Huge $232:0$}};
            \node[c3]{{\Huge $332:3$}};
            [.\node[c1]{{\Huge $432$}};
                \node[c0]{{\Huge $0432:0$}};
                \node[c2]{{\Huge $1432:2$}};
                \node[c4]{{\Huge $2432:4$}};
                [.\node[c1]{{\Huge $3432$}};
                    \node[c3]{{\Huge $03432:3$}};
                    \node[c2]{{\Huge $13432:2$}};
                    [.\node[c1]{{\Huge $23432$}};
                        [.\node[c1]{{\Huge $023432$}};
                            \node[c2]{{\Huge $0023432:2$}};
                            \node[c2]{{\Huge $1023432:2$}};
                            \node[c2]{{\Huge $2023432:2$}};
                            \node[c2]{{\Huge $3023432:2$}};
                            \node[c2]{{\Huge $4023432:2$}};
                        ]
                        [.\node[c1]{{\Huge $123432$}};
                            \node[c1]{{\Huge $0123432$}};
                            \node[c1]{{\Huge $1123432$}};
                            \node[c1]{{\Huge $2123432$}};
                            \node[c1]{{\Huge $3123432$}};
                            \node[c1]{{\Huge $4123432$}};
                        ]
                        [.\node[c1]{{\Huge $223432$}};
                            \node[c1]{{\Huge $0223432$}};
                            \node[c1]{{\Huge $1223432$}};
                            \node[c1]{{\Huge $2223432$}};
                            \node[c1]{{\Huge $3223432$}};
                            \node[c1]{{\Huge $4223432$}};
                        ]
                        [.\node[c1]{{\Huge $323432$}};
                            \node[c2]{{\Huge $0323432:2$}};
                            \node[c2]{{\Huge $1323432:2$}};
                            \node[c2]{{\Huge $2323432:2$}};
                            \node[c2]{{\Huge $3323432:2$}};
                            \node[c2]{{\Huge $4323432:2$}};
                        ]
                        [.\node[c1]{{\Huge $423432$}};
                            \node[c4]{{\Huge $0423432:4$}};
                            \node[c4]{{\Huge $1423432:4$}};
                            \node[c4]{{\Huge $2423432:4$}};
                            \node[c4]{{\Huge $3423432:4$}};
                            \node[c4]{{\Huge $4423432:4$}};
                        ]
                    ]
                    \node[c0]{{\Huge $33432:0$}};
                    \node[c4]{{\Huge $43432:4$}};
                ]
                \node[c3]{{\Huge $4432:3$}};
            ]
        ]
        \node[c4]{{\Huge $42:4$}};
    ]
    \node[c2]{{\Huge $3:2$}};
    \node[c3]{{\Huge $4:3$}};
]
\end{tikzpicture}
\end{center}
\caption{Levels $0$ through $7$ of the tree of behaviors}
\label{fig:tree1}
\end{figure}

\subsection{The Remaining Values of $N$}\label{ss:sporadic}

Theorem~\ref{thm:main} characterizes the behavior of $Q_{\bar N}$ for all
\[
N\notin\st{n:2\leq n\leq34}\cup\st{n:1<n<118\text{ and }n\equiv 1\pb{\modd 5}}\cup\st{57,67,82,107,117}.
\]
These $55$ sequences can be studied individually by generating the sequences and observing the terms.  This study is carried out in~\cite{thesis}; what follows is a summary of those findings.  If $N\leq27$, $Q_{\bar N}$ appears to behave chaotically and last for a long time (at least $10$ million terms), unless $N\in\st{19,23,26}$, in which case $Q_{\bar N}$ is finite with no observable structure.  Thereafter, the remaining $Q_{\bar N}$ sequences are finite, except for $N\in\st{33,36,67,71}$.  These all eventually satisfy the conditions of Proposition~\ref{prop:rst} and are therefore conjectured to be infinite.

Of the $N$ values exceeding $27$, all but $N=67$ and $N=117$ can be computed explicitly until either a $0$ appears (ending the sequence) or until the conditions of Proposition~\ref{prop:rst} are satisfied.  For $N=67$ and $N=117$, Theorem~\ref{thm:main} can be used to describe many terms.  Thereafter, $Q_{\overline{67}}$ can be shown to satisfy Proposition~\ref{prop:rst}, and $Q_{\overline{117}}$ can be shown to reach $0$, and hence its end, not too long after Theorem~\ref{thm:main} stops applying.

\section{Future Work}\label{s:future}
This paper presents an initial application of the method of using parametrized families of initial conditions to generate solutions to nested recurrence relations.  The application here involves the simplest sort of non-constant initial condition (a sequential one) and the most well-known nested recurrence (the $Q$-recurrence).  Our methods are applicable to a wider range of initial conditions and recurrence relations.  Some initial work in this direction is undertaken in~\cite{thesis}.

Subsection~\ref{ss:rst} introduces a seemingly chaotic system of three nested recurrences, and it uses them to construct a solution to Hofstadter's recurrence that weaves together predictability and unpredictability.  There are a few directions that work related to these sequences could progress in~\cite{thesis}.  These include a deeper study of the properties of the $R$, $S$, and $T$ sequences or a search for other families of solutions to the $Q$-recurrence (or other recurrences) that behave analogously to Proposition~\ref{prop:rst}.

Quasilinear sequences appear frequently in this paper.  There are many known solutions to nested recurrences that are eventually quasilinear~\cite{golomb, symbhof}.  The quasilinear chunks of these solutions have a fixed starting point, but they continue forever.  Here, we consider quasilinear chunks with fixed starting and stopping points.  Previous work~\cite{symbhof} gives a method for discovering eventually quasilinear solutions of arbitrary periods to arbitrary recurrences.  Perhaps it can be extended to also find temporary quasilinear solutions.

Finally, the structure of the tree in Figure~\ref{fig:tree1} remains poorly understood.  In particular, it is unknown whether the tree is finite or infinite, and it is unknown whether every string in $\st{0,1,2,3,4}^*$ has a suffix in some leaf (equivalently, whether $j\p{N}<\infty$ for all $N$).  A full understanding of this tree would give a more efficient characterization of the behaviors of the sequences $Q_{\bar N}$.

\appendix

\section{First $28$ terms following initial condition of $Q_N$}\label{app:Qwd}

Assuming $N\geq14$, these are the first $28$ terms of $Q_N$ following the initial condition.

    \begin{align*}
    \mbb{Q_N\p{N+1}}&=Q_N\p{N+1-Q_N\p{N}}+Q_N\p{N+1-Q_N\p{N-1}}\\
    &=Q_N\p{N+1-N}+Q_N\p{N+1-\pb{N-1}}\\
    &=Q_N\p{1}+Q_N\p{2}=1+2=\mbb{3}\\
    &(N\geq2)
    \end{align*}
    \begin{align*}
    \mbb{Q_N\p{N+2}}&=Q_N\p{N+2-Q_N\p{N+1}}+Q_N\p{N+2-Q_N\p{N}}\\
    &=Q_N\p{N+2-3}+Q_N\p{N+2-N}\\
    &=Q_N\p{N-1}+Q_N\p{2}=N-1+2=\mbb{N+1}\\
    &(N\geq2)
    \end{align*}
    \begin{align*}
    \mbb{Q_N\p{N+3}}&=Q_N\p{N+3-Q_N\p{N+2}}+Q_N\p{N+3-Q_N\p{N+1}}\\
    &=Q_N\p{N+3-\pb{N+1}}+Q_N\p{N+3-3}\\
    &=Q_N\p{2}+Q_N\p{N}=2+N=\mbb{N+2}\\
    &(N\geq2)
    \end{align*}
    \begin{align*}
    \mbb{Q_N\p{N+4}}&=Q_N\p{N+4-Q_N\p{N+3}}+Q_N\p{N+4-Q_N\p{N+2}}\\
    &=Q_N\p{N+4-\pb{N+2}}+Q_N\p{N+4-\pb{N+1}}\\
    &=Q_N\p{2}+Q_N\p{3}=2+3=\mbb{5}\\
    &(N\geq3)
    \end{align*}
    \begin{align*}
    \mbb{Q_N\p{N+5}}&=Q_N\p{N+5-Q_N\p{N+4}}+Q_N\p{N+5-Q_N\p{N+3}}\\
    &=Q_N\p{N+5-5}+Q_N\p{N+5-\pb{N+2}}\\
    &=Q_N\p{N}+Q_N\p{3}=N+3=\mbb{N+3}\\
    &(N\geq3)
    \end{align*}
    \begin{align*}
    \mbb{Q_N\p{N+6}}&=Q_N\p{N+6-Q_N\p{N+5}}+Q_N\p{N+6-Q_N\p{N+4}}\\
    &=Q_N\p{N+6-\pb{N+3}}+Q_N\p{N+6-5}\\
    &=Q_N\p{3}+Q_N\p{N+1}=3+3=\mbb{6}\\
    &(N\geq3)
    \end{align*}
    \begin{align*}
    \mbb{Q_N\p{N+7}}&=Q_N\p{N+7-Q_N\p{N+6}}+Q_N\p{N+7-Q_N\p{N+5}}\\
    &=Q_N\p{N+7-6}+Q_N\p{N+7-\pb{N+3}}\\
    &=Q_N\p{N+1}+Q_N\p{4}=3+4=\mbb{7}\\
    &(N\geq4)
    \end{align*}
    \begin{align*}
    \mbb{Q_N\p{N+8}}&=Q_N\p{N+8-Q_N\p{N+7}}+Q_N\p{N+8-Q_N\p{N+6}}\\
    &=Q_N\p{N+8-7}+Q_N\p{N+8-6}\\
    &=Q_N\p{N+1}+Q_N\p{N+2}=3+N+1=\mbb{N+4}\\
    &(N\geq4)
    \end{align*}
    \begin{align*}
    \mbb{Q_N\p{N+9}}&=Q_N\p{N+9-Q_N\p{N+8}}+Q_N\p{N+9-Q_N\p{N+7}}\\
    &=Q_N\p{N+9-\pb{N+4}}+Q_N\p{N+9-7}\\
    &=Q_N\p{5}+Q_N\p{N+2}=5+N+1=\mbb{N+6}\\
    &(N\geq5)
    \end{align*}
    \begin{align*}
    \mbb{Q_N\p{N+10}}&=Q_N\p{N+10-Q_N\p{N+9}}+Q_N\p{N+10-Q_N\p{N+8}}\\
    &=Q_N\p{N+10-\pb{N+6}}+Q_N\p{N+10-\pb{N+4}}\\
    &=Q_N\p{4}+Q_N\p{6}=4+6=\mbb{10}\\
    &(N\geq6)
    \end{align*}
    \begin{align*}
    \mbb{Q_N\p{N+11}}&=Q_N\p{N+11-Q_N\p{N+10}}+Q_N\p{N+11-Q_N\p{N+9}}\\
    &=Q_N\p{N+11-10}+Q_N\p{N+11-\pb{N+6}}\\
    &=Q_N\p{N+1}+Q_N\p{5}=3+5=\mbb{8}\\
    &(N\geq6)
    \end{align*}
    \begin{align*}
    \mbb{Q_N\p{N+12}}&=Q_N\p{N+12-Q_N\p{N+11}}+Q_N\p{N+12-Q_N\p{N+10}}\\
    &=Q_N\p{N+12-8}+Q_N\p{N+12-10}\\
    &=Q_N\p{N+4}+Q_N\p{N+2}=5+N+1=\mbb{N+6}\\
    &(N\geq6)
    \end{align*}
    \begin{align*}
    \mbb{Q_N\p{N+13}}&=Q_N\p{N+13-Q_N\p{N+12}}+Q_N\p{N+13-Q_N\p{N+11}}\\
    &=Q_N\p{N+13-\pb{N+6}}+Q_N\p{N+13-8}\\
    &=Q_N\p{7}+Q_N\p{N+5}=7+N+3=\mbb{N+10}\\
    &(N\geq7)
    \end{align*}
    \begin{align*}
    \mbb{Q_N\p{N+14}}&=Q_N\p{N+14-Q_N\p{N+13}}+Q_N\p{N+14-Q_N\p{N+12}}\\
    &=Q_N\p{N+14-\pb{N+10}}+Q_N\p{N+14-\pb{N+6}}\\
    &=Q_N\p{4}+Q_N\p{8}=4+8=\mbb{12}\\
    &(N\geq8)
    \end{align*}
    \begin{align*}
    \mbb{Q_N\p{N+15}}&=Q_N\p{N+15-Q_N\p{N+14}}+Q_N\p{N+15-Q_N\p{N+13}}\\
    &=Q_N\p{N+15-12}+Q_N\p{N+15-\pb{N+10}}\\
    &=Q_N\p{N+3}+Q_N\p{5}=N+2+5=\mbb{N+7}\\
    &(N\geq8)
    \end{align*}
    \begin{align*}
    \mbb{Q_N\p{N+16}}&=Q_N\p{N+16-Q_N\p{N+15}}+Q_N\p{N+16-Q_N\p{N+14}}\\
    &=Q_N\p{N+16-\pb{N+7}}+Q_N\p{N+16-12}\\
    &=Q_N\p{9}+Q_N\p{N+4}=9+5=\mbb{14}\\
    &(N\geq9)
    \end{align*}
    \begin{align*}
    \mbb{Q_N\p{N+17}}&=Q_N\p{N+17-Q_N\p{N+16}}+Q_N\p{N+17-Q_N\p{N+15}}\\
    &=Q_N\p{N+17-14}+Q_N\p{N+17-\pb{N+7}}\\
    &=Q_N\p{N+3}+Q_N\p{10}=N+2+10=\mbb{N+12}\\
    &(N\geq10)
    \end{align*}
    \begin{align*}
    \mbb{Q_N\p{N+18}}&=Q_N\p{N+18-Q_N\p{N+17}}+Q_N\p{N+18-Q_N\p{N+16}}\\
    &=Q_N\p{N+18-\pb{N+12}}+Q_N\p{N+18-14}\\
    &=Q_N\p{6}+Q_N\p{N+4}=6+5=\mbb{11}\\
    &(N\geq10)
    \end{align*}
    \begin{align*}
    \mbb{Q_N\p{N+19}}&=Q_N\p{N+19-Q_N\p{N+18}}+Q_N\p{N+19-Q_N\p{N+17}}\\
    &=Q_N\p{N+19-11}+Q_N\p{N+19-\pb{N+12}}\\
    &=Q_N\p{N+8}+Q_N\p{7}=N+4+7=\mbb{N+11}\\
    &(N\geq10)
    \end{align*}
    \begin{align*}
    \mbb{Q_N\p{N+20}}&=Q_N\p{N+20-Q_N\p{N+19}}+Q_N\p{N+20-Q_N\p{N+18}}\\
    &=Q_N\p{N+20-\pb{N+11}}+Q_N\p{N+20-11}\\
    &=Q_N\p{9}+Q_N\p{N+9}=9+N+6=\mbb{N+15}\\
    &(N\geq10)
    \end{align*}
    \begin{align*}
    \mbb{Q_N\p{N+21}}&=Q_N\p{N+21-Q_N\p{N+20}}+Q_N\p{N+21-Q_N\p{N+19}}\\
    &=Q_N\p{N+21-\pb{N+15}}+Q_N\p{N+21-\pb{N+11}}\\
    &=Q_N\p{6}+Q_N\p{10}=6+10=\mbb{16}\\
    &(N\geq10)
    \end{align*}
    \begin{align*}
    \mbb{Q_N\p{N+22}}&=Q_N\p{N+22-Q_N\p{N+21}}+Q_N\p{N+22-Q_N\p{N+20}}\\
    &=Q_N\p{N+22-16}+Q_N\p{N+22-\pb{N+15}}\\
    &=Q_N\p{N+6}+Q_N\p{7}=6+7=\mbb{13}\\
    &(N\geq10)
    \end{align*}
    \begin{align*}
    \mbb{Q_N\p{N+23}}&=Q_N\p{N+23-Q_N\p{N+22}}+Q_N\p{N+23-Q_N\p{N+21}}\\
    &=Q_N\p{N+23-13}+Q_N\p{N+23-16}\\
    &=Q_N\p{N+10}+Q_N\p{N+7}=10+7=\mbb{17}\\
    &(N\geq10)
    \end{align*}
    \begin{align*}
    \mbb{Q_N\p{N+24}}&=Q_N\p{N+24-Q_N\p{N+23}}+Q_N\p{N+24-Q_N\p{N+22}}\\
    &=Q_N\p{N+24-17}+Q_N\p{N+24-13}\\
    &=Q_N\p{N+7}+Q_N\p{N+11}=7+8=\mbb{15}\\
    &(N\geq10)
    \end{align*}
    \begin{align*}
    \mbb{Q_N\p{N+25}}&=Q_N\p{N+25-Q_N\p{N+24}}+Q_N\p{N+25-Q_N\p{N+23}}\\
    &=Q_N\p{N+25-15}+Q_N\p{N+25-17}\\
    &=Q_N\p{N+10}+Q_N\p{N+8}=10+N+4=\mbb{N+14}\\
    &(N\geq10)
    \end{align*}
    \begin{align*}
    \mbb{Q_N\p{N+26}}&=Q_N\p{N+26-Q_N\p{N+25}}+Q_N\p{N+26-Q_N\p{N+24}}\\
    &=Q_N\p{N+26-\pb{N+14}}+Q_N\p{N+26-15}\\
    &=Q_N\p{12}+Q_N\p{N+11}=12+8=\mbb{20}\\
    &(N\geq12)
    \end{align*}
    \begin{align*}
    \mbb{Q_N\p{N+27}}&=Q_N\p{N+27-Q_N\p{N+26}}+Q_N\p{N+27-Q_N\p{N+25}}\\
    &=Q_N\p{N+27-20}+Q_N\p{N+27-\pb{N+14}}\\
    &=Q_N\p{N+7}+Q_N\p{13}=7+13=\mbb{20}\\
    &(N\geq13)
    \end{align*}
    \begin{align*}
    \mbb{Q_N\p{N+28}}&=Q_N\p{N+28-Q_N\p{N+27}}+Q_N\p{N+28-Q_N\p{N+26}}\\
    &=Q_N\p{N+28-20}+Q_N\p{N+28-20}\\
    &=Q_N\p{N+8}+Q_N\p{N+8}=N+4+N+4=\mbb{2N+8}\\
    &(N\geq13)
    \end{align*}

\section{Final terms of $Q_{\bar N}$ in the $C_j=0$ case}\label{app:160}

These are the final $158$ terms in $Q_{\bar N}\p{n}$ when $C_j=0$ and $N\geq118$.

\begin{multicols}{2}
\begin{itemize}
\item $Q_{\bar N}\p{A_j+3}=6$
\item $Q_{\bar N}\p{A_j+4}=7$
\item $Q_{\bar N}\p{A_j+5}=8$
\item $Q_{\bar N}\p{A_j+6}=8$
\item $Q_{\bar N}\p{A_j+7}=10$
\item $Q_{\bar N}\p{A_j+8}=A_j\pb{\frac{A_j-A_{j-1}-2}{5}}+B_j+3$
\item $Q_{\bar N}\p{A_j+9}=5$
\item $Q_{\bar N}\p{A_j+10}=8$
\item $Q_{\bar N}\p{A_j+11}=14$
\item $Q_{\bar N}\p{A_j+12}=10$
\item $Q_{\bar N}\p{A_j+13}=11$
\item $Q_{\bar N}\p{A_j+14}=13$
\item $Q_{\bar N}\p{A_j+15}=A_j+7$
\item $Q_{\bar N}\p{A_j+16}=15$
\item $Q_{\bar N}\p{A_j+17}=A_j+10$
\item $Q_{\bar N}\p{A_j+18}=14$
\item $Q_{\bar N}\p{A_j+19}=17$
\item $Q_{\bar N}\p{A_j+20}=14$
\item $Q_{\bar N}\p{A_j+21}=17$
\item $Q_{\bar N}\p{A_j+22}=A_j\pb{\frac{A_j-A_{j-1}-2}{5}}+B_j+11$
\item $Q_{\bar N}\p{A_j+23}=8$
\item $Q_{\bar N}\p{A_j+24}=15$
\item $Q_{\bar N}\p{A_j+25}=A_j+18$
\item $Q_{\bar N}\p{A_j+26}=22$
\item $Q_{\bar N}\p{A_j+27}=17$
\item $Q_{\bar N}\p{A_j+28}=22$
\item $Q_{\bar N}\p{A_j+29}=20$
\item $Q_{\bar N}\p{A_j+30}=A_j\pb{\frac{A_j-A_{j-1}-2}{5}}+B_j+11$
\item $Q_{\bar N}\p{A_j+31}=14$
\item $Q_{\bar N}\p{A_j+32}=14$
\item $Q_{\bar N}\p{A_j+33}=34$
\item $Q_{\bar N}\p{A_j+34}=A_j\pb{\frac{A_j-A_{j-1}-2}{5}}+B_j+14$
\item $Q_{\bar N}\p{A_j+35}=5$
\item $Q_{\bar N}\p{A_j+36}=14$
\item $Q_{\bar N}\p{A_j+37}=22$
\item $Q_{\bar N}\p{A_j+38}=30$
\item $Q_{\bar N}\p{A_j+39}=A_j+15$
\item $Q_{\bar N}\p{A_j+40}=33$
\item $Q_{\bar N}\p{A_j+41}=A_j\pb{\frac{A_j-A_{j-1}-2}{5}}+B_j+29$
\item $Q_{\bar N}\p{A_j+42}=5$
\item $Q_{\bar N}\p{A_j+43}=30$
\item $Q_{\bar N}\p{A_j+44}=A_j+28$
\item $Q_{\bar N}\p{A_j+45}=A_j+24$
\item $Q_{\bar N}\p{A_j+46}=40$
\item $Q_{\bar N}\p{A_j+47}=33$
\item $Q_{\bar N}\p{A_j+48}=A_j\pb{\frac{A_j-A_{j-1}-2}{5}}+A_j+B_j+10$
\item $Q_{\bar N}\p{A_j+49}=15$
\item $Q_{\bar N}\p{A_j+50}=5$
\item $Q_{\bar N}\p{A_j+51}=54$
\item $Q_{\bar N}\p{A_j+52}=36$
\item $Q_{\bar N}\p{A_j+53}=A_j+15$
\item $Q_{\bar N}\p{A_j+54}=53$
\item $Q_{\bar N}\p{A_j+55}=A_j+40$
\item $Q_{\bar N}\p{A_j+56}=22$
\item $Q_{\bar N}\p{A_j+57}=22$
\item $Q_{\bar N}\p{A_j+58}=28$
\item $Q_{\bar N}\p{A_j+59}=36$
\item $Q_{\bar N}\p{A_j+60}=29$
\item $Q_{\bar N}\p{A_j+61}=A_j+32$
\item $Q_{\bar N}\p{A_j+62}=64$
\item $Q_{\bar N}\p{A_j+63}=36$
\item $Q_{\bar N}\p{A_j+64}=A_j\pb{\frac{A_j-A_{j-1}-2}{5}}+B_j+22$
\item $Q_{\bar N}\p{A_j+65}=20$
\item $Q_{\bar N}\p{A_j+66}=40$
\item $Q_{\bar N}\p{A_j+67}=50$
\item $Q_{\bar N}\p{A_j+68}=36$
\item $Q_{\bar N}\p{A_j+69}=51$
\item $Q_{\bar N}\p{A_j+70}=A_j\pb{\frac{A_j-A_{j-1}-2}{5}}+B_j+31$
\item $Q_{\bar N}\p{A_j+71}=14$
\item $Q_{\bar N}\p{A_j+72}=28$
\item $Q_{\bar N}\p{A_j+73}=A_j+60$
\item $Q_{\bar N}\p{A_j+74}=54$
\item $Q_{\bar N}\p{A_j+75}=32$
\item $Q_{\bar N}\p{A_j+76}=A_j\pb{\frac{A_j-A_{j-1}-2}{5}}+A_j+B_j+39$
\item $Q_{\bar N}\p{A_j+77}=A_j+24$
\item $Q_{\bar N}\p{A_j+78}=54$
\item $Q_{\bar N}\p{A_j+79}=A_j+73$
\item $Q_{\bar N}\p{A_j+80}=29$
\item $Q_{\bar N}\p{A_j+81}=44$
\item $Q_{\bar N}\p{A_j+82}=A_j+45$
\item $Q_{\bar N}\p{A_j+83}=A_j+53$
\item $Q_{\bar N}\p{A_j+84}=70$
\item $Q_{\bar N}\p{A_j+85}=A_j+39$
\item $Q_{\bar N}\p{A_j+86}=62$
\item $Q_{\bar N}\p{A_j+87}=A_j+66$
\item $Q_{\bar N}\p{A_j+88}=44$
\item $Q_{\bar N}\p{A_j+89}=A_j+47$
\item $Q_{\bar N}\p{A_j+90}=83$
\item $Q_{\bar N}\p{A_j+91}=A_j\pb{\frac{A_j-A_{j-1}-2}{5}}+B_j+47$
\item $Q_{\bar N}\p{A_j+92}=5$
\item $Q_{\bar N}\p{A_j+93}=44$
\item $Q_{\bar N}\p{A_j+94}=A_j+52$
\item $Q_{\bar N}\p{A_j+95}=97$
\item $Q_{\bar N}\p{A_j+96}=49$
\item $Q_{\bar N}\p{A_j+97}=2A_j\pb{\frac{A_j-A_{j-1}-2}{5}}+A_j+2B_j+10$
\item $Q_{\bar N}\p{A_j+98}=15$
\item $Q_{\bar N}\p{A_j+99}=70$
\item $Q_{\bar N}\p{A_j+100}=A_j\pb{\frac{A_j-A_{j-1}-2}{5}}+A_j+B_j+50$
\item $Q_{\bar N}\p{A_j+101}=14$
\item $Q_{\bar N}\p{A_j+102}=44$
\item $Q_{\bar N}\p{A_j+103}=A_j+83$
\item $Q_{\bar N}\p{A_j+104}=50$
\item $Q_{\bar N}\p{A_j+105}=A_j+62$
\item $Q_{\bar N}\p{A_j+106}=66$
\item $Q_{\bar N}\p{A_j+107}=A_j\pb{\frac{A_j-A_{j-1}-2}{5}}+B_j+74$
\item $Q_{\bar N}\p{A_j+108}=5$
\item $Q_{\bar N}\p{A_j+109}=50$
\item $Q_{\bar N}\p{A_j+110}=A_j+91$
\item $Q_{\bar N}\p{A_j+111}=A_j+52$
\item $Q_{\bar N}\p{A_j+112}=81$
\item $Q_{\bar N}\p{A_j+113}=75$
\item $Q_{\bar N}\p{A_j+114}=A_j+49$
\item $Q_{\bar N}\p{A_j+115}=99$
\item $Q_{\bar N}\p{A_j+116}=A_j+77$
\item $Q_{\bar N}\p{A_j+117}=54$
\item $Q_{\bar N}\p{A_j+118}=A_j\pb{\frac{A_j-A_{j-1}-2}{5}}+B_j+63$
\item $Q_{\bar N}\p{A_j+119}=20$
\item $Q_{\bar N}\p{A_j+120}=A_j\pb{\frac{A_j-A_{j-1}-2}{5}}+A_j+B_j+50$
\item $Q_{\bar N}\p{A_j+121}=14$
\item $Q_{\bar N}\p{A_j+122}=5$
\item $Q_{\bar N}\p{A_j+123}=A_j\pb{\frac{A_j-A_{j-1}-2}{5}}+B_j+113$
\item $Q_{\bar N}\p{A_j+124}=20$
\item $Q_{\bar N}\p{A_j+125}=A_j+62$
\item $Q_{\bar N}\p{A_j+126}=130$
\item $Q_{\bar N}\p{A_j+127}=A_j+65$
\item $Q_{\bar N}\p{A_j+128}=66$
\item $Q_{\bar N}\p{A_j+129}=100$
\item $Q_{\bar N}\p{A_j+130}=2A_j\pb{\frac{A_j-A_{j-1}-2}{5}}+2B_j+33$
\item $Q_{\bar N}\p{A_j+131}=14$
\item $Q_{\bar N}\p{A_j+132}=A_j\pb{\frac{A_j-A_{j-1}-2}{5}}+B_j+63$
\item $Q_{\bar N}\p{A_j+133}=20$
\item $Q_{\bar N}\p{A_j+134}=A_j+49$
\item $Q_{\bar N}\p{A_j+135}=185$
\item $Q_{\bar N}\p{A_j+136}=92$
\item $Q_{\bar N}\p{A_j+137}=2A_j+24$
\item $Q_{\bar N}\p{A_j+138}=40$
\item $Q_{\bar N}\p{A_j+139}=70$
\item $Q_{\bar N}\p{A_j+140}=2A_j\pb{\frac{A_j-A_{j-1}-2}{5}}+A_j+2B_j+81$
\item $Q_{\bar N}\p{A_j+141}=14$
\item $Q_{\bar N}\p{A_j+142}=66$
\item $Q_{\bar N}\p{A_j+143}=A_j+124$
\item $Q_{\bar N}\p{A_j+144}=74$
\item $Q_{\bar N}\p{A_j+145}=35$
\item $Q_{\bar N}\p{A_j+146}=A_j+80$
\item $Q_{\bar N}\p{A_j+147}=148$
\item $Q_{\bar N}\p{A_j+148}=A_j\pb{\frac{A_j-A_{j-1}-2}{5}}+B_j+68$
\item $Q_{\bar N}\p{A_j+149}=5$
\item $Q_{\bar N}\p{A_j+150}=35$
\item $Q_{\bar N}\p{A_j+151}=2A_j+157$
\item $Q_{\bar N}\p{A_j+152}=54$
\item $Q_{\bar N}\p{A_j+153}=70$
\item $Q_{\bar N}\p{A_j+154}=A_j\pb{\frac{A_j-A_{j-1}-2}{5}}+A_j+B_j+120$
\item $Q_{\bar N}\p{A_j+155}=A_j+39$
\item $Q_{\bar N}\p{A_j+156}=117$
\item $Q_{\bar N}\p{A_j+157}=151$
\item $Q_{\bar N}\p{A_j+158}=A_j\pb{\frac{A_j-A_{j-1}-2}{5}}+B_j+39$
\item $Q_{\bar N}\p{A_j+159}=A_j\pb{\frac{A_j-A_{j-1}-2}{5}}+B_j+3$
\item $Q_{\bar N}\p{A_j+160}=0$
\end{itemize}
\end{multicols}

\section*{Acknowledgements}
I would like to thank Dr.\ Doron Zeilberger of Rutgers University for introducing me to the Hofstadter $Q$-recurrence, and to him along with Dr.\ Michael Saks, Dr.\ Swastik Kopparty, and Dr.\ Neil Sloane (my dissertation committee) for their feedback on this work. I would also like to thank Yonah Biers-Ariel of Rutgers University for proofreading a draft of this paper and providing me with useful feedback.

\end{document}